\documentclass[]{interact}

\usepackage{epstopdf}
\usepackage{subfigure}

\usepackage{natbib}
\bibpunct[, ]{(}{)}{;}{a}{}{,}
\renewcommand\bibfont{\fontsize{10}{12}\selectfont}
\usepackage{url}

\theoremstyle{plain}
\newtheorem{theorem}{Theorem}[section]
\newtheorem{lemma}[theorem]{Lemma}
\newtheorem{corollary}[theorem]{Corollary}

\theoremstyle{definition}

\newtheorem{example}[theorem]{Example}

\theoremstyle{remark}

\usepackage{amsthm,amsmath,amssymb}
\usepackage{mathrsfs}

\usepackage{soul}
\usepackage{color, xcolor}
\usepackage{booktabs}
\usepackage{longtable}

\soulregister\cite7
\soulregister\ref7 
\soulregister\cite7 
\soulregister\citep7 
\soulregister\citet7 
\soulregister\ref7 
\soulregister\pageref7 

\begin{document}

\articletype{Original Article}

\title{Improved formulations of the joint order batching and picker routing problem}

\author{
\name{Kai Zhang\textsuperscript{a} and Chuanhou Gao\textsuperscript{a}\thanks{CONTACT C.H. Gao. Author. Email: gaochou@zju.edu.cn}}
\affil{\textsuperscript{a}School of Mathematical Sciences, Zhejiang University, Hangzhou, China}
}

\maketitle

\begin{abstract}
Order picking is the process of retrieving ordered products from storage locations in warehouses. \hl{In picker-to-parts order picking systems, two or more customer orders may be grouped and assigned to a single picker. Then routing decision regarding the visiting sequence of items during a picking tour must be made. \citep{doi:10.1080/00207540410001733896} found that solving the integrated problem of batching and routing enables warehouse managers to organize order picking operations more efficiently compared with solving the two problems separately and sequentially. We therefore investigate the mathematical programming formulation of this integrated problem.}

We present several improved formulations for the \hl{problem based on the findings} of \citep{VALLE2017817}, \hl{that} can significantly improve computational results. \hl{More specifically, we reconstruct the connectivity constraints and generate new cutting planes in our branch-and-cut framework. We also discuss some problem properties by studying the structure of the graphical representation, and we present two types of additional constraints. We also consider the no-reversal case of this problem. We present efficient formulations by building different auxiliary graphs.} Finally, we present computational results for publicly available test problems for \hl{single-block and multiple-block} warehouse configurations.
\end{abstract}


\section{Introduction} \label{intro}

In modern business environments, warehousing and relative order picking processes are essential components of any supply chain \citep{Warehousing@Manzini2012}. Order picking is the process of retrieving ordered products from storage locations in warehouses, \hl{and it} typically accounts for 55\% of the total warehouse operating expense and has long been recognized as the most labor-intensive and costly activity for warehouses. Therefore, the order picking process should be robustly designed and optimally controlled to handle requirements efficiently \citep{DEKOSTER2007481,Tompkin2010}. 

\hl{The order picking process should be investigated within a system context. In picker-to-parts systems, pickers walk or ride through the picking area to collect the requested items. In parts-to-picker systems, automated cranes move along the aisle, retrieve unit loads, and bring them to a pick position. This study focuses on order picking operations in picker-to-parts systems, which still account for the large majority of all order picking systems \citep{DEKOSTER2007481,doi:10.1080/13675567.2014.945400,VANGILS20181}. 
}

\hl{According to \citep{VANGILS20181}, decisions to manage order picking in picker-to-parts systems can be classified into strategic, tactical or operational decisions. In this paper we focus on the operational planning problems that typically concern daily operations. Operational planning problems include (1) how batches of orders are generated (batching), (2) how each picker is routed (routing), (3) how the daily required number of order pickers is determined (workforce level), (4) how a given workforce is allocated to the picking and sorting operations (workforce allocation), and (5) how the picking orders are sequenced (job assignment).} Some researchers focus on individual planning problems, \hl{have developed} many efficient algorithms, and have successfully applied \hl{theses algorithms} to suboptimal problems. However, \hl{simultaneously} optimizing multiple order picking planning problems \hl{can} result in even more efficient picking operations. \citep{VANGILS20181} presented a comprehensive review of research on combination problems in order picking systems. According to their study, the joint order batching and picker routing problem (JOBPRP), which is also the focus of attention in this study, has received more attention than other combination problems (for example, the integrated problem of order batching and job assignment). \hl{One of the first works that considers this integrated problem was \citep{doi:10.1080/00207540410001733896}. The crucial observation from their simulation experiment is that a simultaneous solution yields significantly better performance benefits than a sequential solution. More recently, \citep{Scholz2017} integrated different routing algorithms into a heuristic approach for the batching problem. Their numerical experiments demonstrated the benefits from solving the joint problem.} 

\hl{We now introduce the JOBPRP from practical application viewpoints. In picker-to-parts systems, an order picker is guided by a pick list that comprises one or more customer orders on his/her picking tour. A customer order typically requires a list of distinct products, and customer orders can be converted into a pick list until the capacity of the picking device is exhausted. To minimize the total travel distance, order pickers have to decide how the orders should be assigned to picking tours, which give rise to the so-called order batching problem. For each picking tour, the shortest path to visit a set of picking locations should be determined, which give rise to the so-called picker routing problem. Those two problems are thought to be strongly linked, because solving the routing problem is dependent on the solution of the batching problem. Obviously, solving the integrated problem, the JOBPRP, can improve the efficiency of order pickers.}

\hl{Although} there are many heuristic-solving frameworks, very few exact algorithms for the JOBPRP have been proposed in the literature. \citep{VALLE2017817} proposed a novel exact algorithm that incorporates a non-compact integer programming formulation and a branch-and-cut procedure. The original model proposed by \citep{Valle2016} can be significantly improved by adding valid inequalities based on the standard layout of warehouses. Inspired by \citep{VALLE2017817}, we analyze existing models for the JOBPRP and propose modified solution approaches to improve computational performance. The following are the main \hl{contributions} of our study:

\hl{(1) We reconstruct the connectivity constraints by making full use of the properties of a rectangular warehouse with multiple blocks. The improved formulations achieve computational efficiency.

(2) We also discuss some
problem properties by studying the structure of a graphical representation, and we present two types of additional constraints.} 

(3) \hl{We also consider} the no-reversal JOBPRP. We propose a traveling salesman problem (TSP) formulation for single-block and 2-block warehouses by introducing two auxiliary graphs. 

(4) We conduct a series of numerical experiments to evaluate our formulations. The test instances used here are generated using \hl{the method provided by} \citep{Valle2016}.

The remainder of this \hl{paper} is organized as follows. Section~\ref{review} comprises a literature review regarding the JOBPRP and some closely related problems. Section~\ref{prodes} briefly introduces the JOBPRP and presents a graph-based formulation for this problem. Section~\ref{imp_formulation} presents two improved formulations, and Section~\ref{additional_cons} introduces two types of additional constraints to improve solution \hl{quality}. Section~\ref{no-revcase} discusses the no-reversal JOBPRP in detail. Section~\ref{computrel} reports on some computational results and observations, and Section~\ref{conclusion} concludes the study.

\section{Literature review}\label{review}

Picker routing and order batching problems have received considerable interest since the 1980s \citep{doi:10.1080/00207548108956683,Vannelli1986,10.1287/opre.31.3.507,tspp1985}. In this section, we provide a summary of some previous studies \hl{on} order batching and picker routing problems. 

\subsection{Picker routing problem}\label{prp}

This problem can be solved to optimality using any exact approach to the TSP. However, more efficient solution approaches can be obtained using \hl{a} particular warehouse layout. The first attempt to provide a problem-specific exact approach to picker routing problems was proposed \hl{by} \citep{10.1287/opre.31.3.507}. They constructed a sparse graph representation for a rectangular warehouse containing a single block and presented a polynomial-time dynamic programming algorithm. \hl{Researchers then developed two generalized algorithms based} on the work of \citep{10.1287/opre.31.3.507}. \hl{\citep{tspp1985} interpreted the routing problem as a Steiner travelling salesman problem (STSP) and extended the Ratliff-Rosenthal algorithm to all series-parallel graphs; \citep{doi:10.1080/00207540110028128} modified the Ratliff-Rosenthal algorithm and introduced routing heuristics for 2-block and more complex layouts.} Recently, \citep{SCHOLZ201668} proposed an exact approach regarding the unique structure of a single-block warehouse. They introduced integer programming whose size is independent of the number of picking locations \hl{and} demonstrated that this formulation can significantly improve computational performance.

Because the TSP is NP-hard, optimal routing is often regarded as difficult to determine. Many heuristics have been proposed for the problem from a practical standpoint based on different routing strategies, including S-shape \citep{doi:10.1080/07408178808966150}, midpoint\citep{doi:10.1080/07408179308964306}, largest gap \citep{doi:10.1080/07408179308964306}, combined \citep{Petersen1997} and aisle-by-aisle \citep{doi:10.1080/002075499191580}. TSP heuristics can also be used to solve the picker routing problem. \citep{THEYS2010755} discovered that the Lin--Kernighan--Helsgaun heuristic \citep{HELSGAUN2000106} outperforms the S-shape heuristic when there are two or more blocks in the warehouse. \citep{CAMBAZARD2018419} developed a dynamic programming approach for \hl{a} rectilinear TSP, \hl{and} the algorithm is also applicable to the picker routing problem. However, the complexity grows exponentially with the number of blocks. \hl{Other heuristics are created from metaheuristics}. For example, \citep{HoTseng2006} proposed a simulating annealing heuristic which is integrated with the largest gap routing strategy. The reader interested in the picker routing problem may refer to \citep{MASAE2020107564} for a comprehensive review. 

\subsection{Order batching problem}\label{obp}

\hl{The order batching problem can be formally defined as follows: How, given the capacity of the picking device and the adopted routing strategy, can a given set of customer orders with known storage locations be grouped into picking orders such that the total lengths of all picker tours is minimized? \citep{Wscher2004OrderPA,Warehousing@Manzini2012}
}

\hl{Because the order batching problem is known to be NP-hard \citep{doi:10.1080/07408170108936837}, exact approaches are typically impractical for instances with relatively large sizes. As a result, many scholars consider heuristics and metaheuristics. According to \citep{Warehousing@Manzini2012}, batching heuristics can be distinguished into savings, seed, or priority rule-based algorithms as well as other algorithms. Savings algorithms are based on the algorithm of \citep{doi:10.1287/opre.12.4.568} for the vehicle routing problem. The initial version of the savings algorithm for
the batching problem can be described as follows: savings are computed in terms of reducing the travel distance by collecting items for two customer orders on a single picking tour instead of collecting them separately, and then, orders are sequentially assigned to
batches based on the savings (e.g., see \citep{doi:10.1080/00207548908942610,doi:10.1080/00207540600920850}). Meanwhile, the seed algorithm introduced by \citep{doi:10.1080/00207548108956683} generates batches by
means of a two-phase procedure: a seed order
is first selected and added to a new batch according to a seed selection rule, and then, unassigned orders are added to this batch according to an order addition rule (e.g., see \citep{GIBSON199257,doi:10.1080/00207549608904926,doi:10.1080/00207540600558015}). The priority rule-based algorithm also consists of a two-step procedure: first, priorities are assigned to the customer orders, and then, customer orders are assigned successively to batches in the sequence given by the priorities (e.g., see \citep{PAN1995691,doi:10.1287/mnsc.45.4.575}). There are also some metaheuristics for the order batching problem. The interested reader may refer to \citep{Warehousing@Manzini2012,VANGILS20181,Cergibozan2019} for further details.
}

\subsection{Joint order batching and picker routing problem}\label{joint_rb}

Considering the strong relationship between batching and routing, solving these planning problems in a detailed manner would be beneficial. In recent decades, many efficient heuristic and metaheuristic methods have been proposed to solve the JOBPRP.

\citep{doi:10.1080/00207540410001733896} was one of the first to formulate the batching and routing problem jointly as a combinatorial optimization problem. \hl{Their} proposed two-step heuristic first constructs batches sequentially and then solves the subsequent routing problem. \hl{\citep{doi:10.1080/0740817X.2011.588994} presented a route-selection based formulation that enumerates all possible routes and compared their heuristic solution with a lower bound developed by a relaxation model. \citep{Kulak2012} proposed a tabu search algorithm integrated with a clustering algorithm that generates an initial solution. They also proposed two constructive heuristics to solve the picker routing problem. \citep{RePEc:dar:wpaper:65331} developed a simulated annealing algorithm to determine order batches and picker routes and applied four different heuristics to form initial order batches.} \citep{CHENG2015805} proposed a hybrid approach consisting of a particle swarm optimization for batching, whereas \citep{doi:10.1080/00207543.2016.1187313} proposed a constructive heuristic based on similarity coefficient for batching; both used an ant colony optimization algorithm in the routing procedure. \citep{Scholz2017} introduced an iterated local search algorithm, which allows for integrating different routing algorithms. \citep{ARBEXVALLE2020460} presented an approximate formulation for this problem. They also proposed a partial integer optimization heuristic based on their formulation. \citep{BRIANT2020497} proposed a heuristic based on column generation to deal with an exponential linear programming formulation of the JOBPRP. \citep{AERTS2021105168} modeled the JOBPRP as a clustered vehicle routing problem and applied a two-level variable neighborhood search algorithm developed by \citep{DEFRYN201778}. \citep{doi:10.1080/00207543.2020.1766712} presented a model of JOBRPR under uncertainty; metaheuristics, such as genetic, particle swarm optimization, and artificial honeybee colony algorithms are used as approaches to solve the formulated model.

\hl{To enhance efficiency and customer service, some researchers also take the due dates of the customer orders into account, which initiates the picking sequencing problem. \citep{Tsai2008} considered earliness and tardiness penalties and suggested a genetic algorithm under the assumption that splitting customer orders is allowed. \citep{CHEN2015158} developed a genetic algorithm for the order batching and sequencing processes. For the routing decision for each batch, they adopted an ant colony algorithm. \citep{SCHOLZ2017461} introduced a mixed-integer linear formulation whose size increases polynomially with the number of orders. They also proposed a variable neighborhood descent algorithm that could work with very large problems. Meanwhile, \citep{VANGILS2019814} proposed an iterated local search algorithm to solve the problem effectively and efficiently. They also showed the substantial performance benefits gained from integrating planning problems using a real-life case study.
}

Apart from these heuristic and metaheuristic methods, only a few exact approaches have been proposed in the literature. \citep{Valle2016} presented three basic formulations of the JOBPRP; one of them involves exponentially many constraints and the remaining two are based on network flows. They used the branch-and-cut algorithm presented by \citep{10.1137/1033004} for the first formulation. A JOBPRP-test instance generator based on publicly available real-world data was also introduced. The non-compact formulations proposed by \citep{Valle2016} was improved by \citep{VALLE2017817}. They introduced a significant number of valid inequalities to strengthen the linear relaxation of their formulation. 

\section{Problem description and basic formulation}\label{prodes}

\hl{In this section, we provide more background information regarding the warehouse layout. We then introduce a basic formulation of the JOBPRP.
}
\subsection{Background information}

\hl{We consider a rectangular warehouse with a manual picker-to-parts order-picking system.} The warehouse is composed of one origin, several vertical picking aisles, and several horizontal cross-aisles. We call the part between two adjacent cross-aisles a block, and the section of a picking aisle within a block a subaisle. If a warehouse has $q$ blocks, any picking aisle in the warehouse can be partitioned into $q$ subaisles. We illustrate these concepts in Figure~\ref{warehouse_layout}.

\begin{figure}
\centering
\resizebox*{8.5cm}{!}{\includegraphics{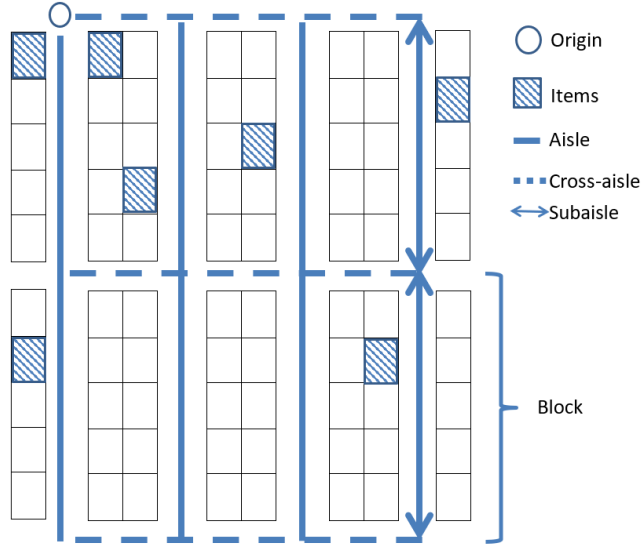}}
\hspace{5pt}
\caption{Example of a 2-block warehouse layout} \label{warehouse_layout}
\end{figure}

\hl{Each vertical aisle contains a set of picking locations on both sides and a picking location contains several storage slots. We assume that each slot holds one type of product and each product type is assigned to only one slot. In addition, products are divided into several classes, and products belonging to the same class are placed in consecutive slots. As a result, most items in a subaisle belong to the same class. The adopted storage assignment policy is actually a variant of the class-based storage policy \citep{Warehousing@Manzini2012}.
}

\hl{Before a shift starts, the number of order pickers available for carrying out the picking operations has been determined. Pickers start at the origin, visit a set of picking locations to retrieve the order products and return to the origin. During a picking tour, a picker is equipped with a trolley that accommodates a limited number of baskets. The necessary number of baskets to carry each customer order is assumed to be known. A customer order typically requires a list of distinct products, and we assume that a customer order cannot be split over various batches. The reason is that mixing and dividing orders can result in an unacceptable consolidation effort \citep{Valle2016,Scholz2017}. 
}

\subsection{Model formulation}

\hl{The basic formulation of the JOBPRP that we discuss in this section is based on the STSP. Thus we first introduce a formulation for the STSP in the context of the single-picker routing problem. Then, we provide the basic formulation for the JOBPRP.
}

Let $V_L$ denote a set of picking locations, each of which is on a subaisle and contains one or more requested products. Let $V_I$ denote a set of endpoints of each subaisle, and we call these points 'artificial locations'. For simplicity, we assume that the origin $s$ is located at the
top left corner of the warehouse, and it just overlaps the first artificial location. \hl{We define a sparse graph using the vertex set $V$ and the edge set $E$. Edges in $E$ connect the following pair of vertices: (1) two neighboring locations within a picking aisle and (2) two neighboring artificial locations within a cross-aisle.} The length $d_e$ of any edge $e\in E$ is equal to the direct distance between two locations. The graphical representation of a warehouse is shown in Figure~\ref{pointset}.

\begin{figure}
\centering
\resizebox*{8.5cm}{!}{\includegraphics{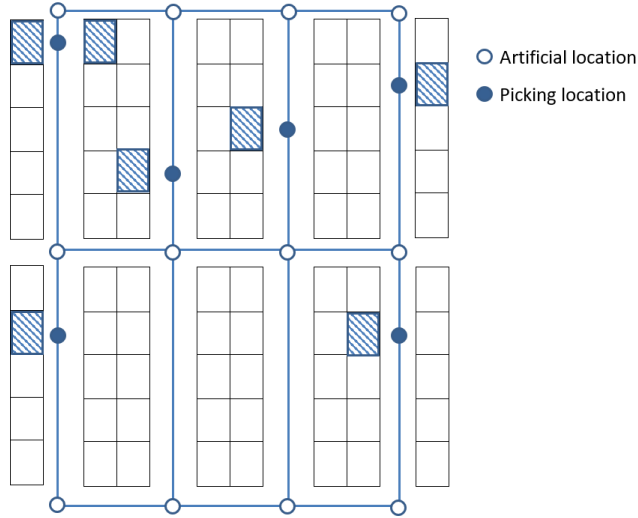}}
\hspace{5pt}
\caption{The related graph optimization problem} \label{pointset}
\end{figure}

Now we \hl{present a non-compact formulation, inspired by the work of \citep{LETCHFORD201383}, for} the single-picker routing problem. Before a picker enters the warehouse, the graphical representation $G=(V,E)$ is already known. The task is to find a \hl{closed} walk by which every $v\in V_L$ is visited. Note that the walk \hl{need} not be Hamiltonian or Euler circuits, that is, a vertex or an edge can be visited more than once by the walk. Let $[u,v]$ denote the unordered pair of location $u$ and location $v$, i.e., $[u,v]$ is the edge connecting $u$ and $v$. For any node set $S \subset V$, $\delta(S)$ denotes the set of edges with exactly one end-node inside $S$. For a single vertex $v \in V$, let $\delta(v) = \delta(\{ v \})$. We introduce a nonnegative decision variable $x_e \in \mathbb{Z}$ to represent the number of times edge $e$ is traversed. We also use a binary decision variable $y_v$ to indicate whether vertex $v \in V \backslash \{ s\}$ is visited by the walk. The single-picker routing problem can be easily described by the following program.
\begin{alignat}{2} 
\label{sp0} \quad \min \quad & \sum_{e\in E} d_e x_e & \qquad & \\
\label{sp1} \mbox {s.t.} \quad & \sum_{e\in \delta(v)}x_e \geq 1, & \qquad & \forall v \in \{ s\} \cup V_L \\ 
\label{sp2} \qquad & y_v \geq \min \{x_e,1\}, & \qquad &  \forall v\in V\backslash \{ s \}, e \in \delta(v)\\
\label{sp3} \qquad & \sum_{e\in \delta(S)} x_e \geq y_v,  & \qquad & \forall v \in S, S\subset V\backslash \{ s\}, |S| \geq 2\\
\label{sp4} \qquad & \sum_{e\in \delta(v)} x_e \colorbox{yellow}{is an even integer}, & \qquad & \forall v \in V\\
\label{sp5} \qquad & x_e \in \{ 0,1,2,...\}, & \qquad &  \forall e \in E\\
\label{sp6} \qquad & y_v \in \{ 0,1 \}, & \qquad &  \forall v \in V
\end{alignat}

Constraints~(\ref{sp1}) ensure that each picking location is visited by the picker. Constraints~(\ref{sp2}) define the $y$ variables for each vertex. Constraints~(\ref{sp3}) guarantee that the multigraph induced by the walk is connected. In \hl{the} TSP, constraints~(\ref{sp3}) are also known as subtour elimination constraints. Constraints~(\ref{sp2}),~(\ref{sp3}), and~(\ref{sp4}) ensure that Euler circuit exists in the multigraph induced by the walk. 

\hl{The abovementioned formulation is treated as a starting point for the JOBPRP. The basic formulation for the JOBPRP is constructed by including the assignment of orders to batches as additional constraints. We start by creating a directed graph $\tilde{G}=(V,\tilde{E})$ from graph $G=(V,E)$: any edge $e=[u,v]\in E$ is replaced with two directed arcs $e_1=(u,v)$ and $e_2=(v,u)$. A formulation based on graph $\tilde{G}$ for the single-picker routing problem can be formulated similarly.} Constraints~(\ref{sp4}) will then be replaced by flow constraints, which are known to be totally unimodular constraints. Furthermore, a very useful theorem proposed in \citep{Valle2016} is as follows.

\begin{theorem}\label{onlyonce} Each directed arc in $\tilde{E}$ can only be traversed once by any optimal walk. 
\end{theorem}

When multiple pickers and pick capacity are considered, we should assign orders to pickers. Assume \hl{that} a picker has a capacity of $B$ units, the set of all orders is denoted by $O$, and any order $o \in O$ has a capacity of $b_o$ units. Two or more customer orders can be batched together if the total capacity of customer orders assigned to a picker does not exceed its available capacity. After the order batching process, $O$ will be partitioned into several subsets, and each subset of orders will be assigned to one picker. Each order $o$ contains a subset of picking locations $L_o \subset V$. Let $O_t \subset O$ denote the subset of orders assigned to picker $t$, then picker $t$ must visit all nodes in $\cup_{o\in O_t}L_o$ to collect a set of items for orders in $O_t$. 

Now we present a basic formulation of the JOBPRP based on formulation~(\ref{sp0})-(\ref{sp6}) and Theorem~\ref{onlyonce}. For any node set $S\subset V$, let $\delta^{+}(S)=\{(u,v)\in \tilde{E}: u\in S,v\notin S\}$ and $\delta^{-}(S)=\{(u,v)\in \tilde{E}: u\notin S,v\in S\}$. For a single vertex $v\in V$, let $\delta^{+}(v) = \delta^{+}(\{ v \})$ and $\delta^{-}(v) = \delta^{-}(\{ v \})$. Let $T$ be the number of available pickers, and let $\mathcal{T}=\{1,2,...,T\}$. We introduce \hl{the} binary variables $x_{tuv}$ to indicate whether arc $(u,v)$ is traversed by picker $t$, $y_{tv}$ to indicate whether vertex $v$ is visited by picker $t$ and $z_{ot}$ to indicate whether picker $t$ picks order $o$. The basic formulation is formally given as follows.
\begin{alignat}{2} 
\label{bs0} \quad \min \quad &\sum_{t=1}^{T} \sum_{(u,v)\in \tilde{E}} d_{uv} x_{tuv} & \qquad & \\
\label{bs1} \mbox {s.t.} \quad & \sum_{(s,v)\in \delta^{+}(s)}x_{tsv} \geq 1, & \qquad & \forall t\in \mathcal{T} \\
\label{bs2} \quad & \sum_{(u,v)\in \delta^{+}(u)}x_{tuv} \geq z_{ot}, &  \qquad & \forall t\in \mathcal{T}, o \in O, u \in L_o \\
\label{bs3} \quad & y_{tu} \geq x_{tuv}, & \qquad & \forall t\in \mathcal{T}, u\in V\backslash \{ s \}, (u,v) \in \delta^+(u) \\
\label{bs4} \quad & \sum_{(u,v) \in \delta^+(S)} x_{tuv} \geq y_{tu_0}, & \qquad & \forall t\in \mathcal{T},  S \subset V\backslash \{ s \}, |S|\geq 2, u_0 \in S \\
\label{bs5} \quad & \sum_{(v,u) \in \delta^+(v)} x_{tvu} = \sum_{(u,v) \in \delta^-(v)} x_{tuv}, & \qquad & \forall t\in \mathcal{T}, v\in V \\
\label{bs6} \quad & \sum_{t \in \mathcal{T}} z_{ot} = 1, & \qquad & \forall o \in O \\
\label{bs7} \quad & \sum_{o \in O} b_o z_{ot} \leq B, & \qquad & \forall t \in \mathcal{T} \\
\label{bs8} \qquad & x_{tuv} \in \{ 0,1 \}, & \qquad &  \forall t \in \mathcal{T}, (u,v) \in \tilde{E}\\
\label{bs9} \qquad & y_{tv} \in \{ 0,1 \}, & \qquad &  \forall t \in \mathcal{T}, v \in V\\
\label{bs10} \qquad & z_{ot} \in \{ 0,1 \}, & \qquad &  \forall t \in \mathcal{T}, o \in O
\end{alignat}

\hl{Constraints~(\ref{bs1})-(\ref{bs5}) are similar to constraints~(\ref{sp1})-(\ref{sp4}); each of them describes a picker routing process.} Constraints~(\ref{bs6}) ensure that each order is assigned to precisely one picker. Constraints~(\ref{bs7}) \hl{are} capacity constraints. Constraints~(\ref{bs6})-(\ref{bs7}) describe the order batching process. \hl{The feasible region of the basic formulation is denoted by $P_{basic}$, i.e.,} $P_{basic} = \{ (x,y,z): constraints~(\ref{bs1})-(\ref{bs10})\}$. \hl{This formulations is non-compact because it involves exponentially many constraints~(\ref{bs4}) to enforce connectivity. A branch-and-cut algorithm that separates these constraints should be implemented when this formulation is adopted.}

\hl{The basic formulation is a slightly different version of the original formulation \citep{Valle2016}. The original formulation assumes that it is unnecessary for a picker to depart from the origin when no order is assigned to it, but constraints~(\ref{bs1}) force all pickers to depart from the origin. The main reasons for this assumption are the following:
}
\begin{enumerate}
    \item We do not tackle the workforce level planning problem in this paper, and this assumption helps in simplifying the formulation.
    \item This is not a critical assumption because constraints~(\ref{bs1}) can easily be modified to deal with the previous assumption.
\end{enumerate}
\hl{In the rest of the paper, we simply let $|\mathcal{T}|$ be the necessary number of pickers to carry all products, which can be obtained by solving a bin-packing problem.
}

\section{Improved formulations}\label{imp_formulation}

In this section, we propose two improved formulations $P_G$ and $P_F$ for the JOBPRP by reformulating the connectivity constraints~(\ref{bs4}). \hl{The key idea of reformulation is to enforce connectivity using the subaisle cuts \citep{VALLE2017817}. To improve readability, we present two tables that summarize the most important notations and formulations in the appendix.} 

\hl{We first introduce the subaisle cuts and discuss the relationships between these cuts and the basic formulation.} Let the number of subaisles be $W_{sub}$. Subaisles are indexed by the elements in set $[W_{sub}]=\{1,2,...,W_{sub}\}$. Let the set of picking locations within subaisle $i$ be $V_{sub}(i)$. The northern artificial location is denoted by $f(i)$ and the southern artificial location is denoted by $l(i)$. For any picking location $v \in V_{sub}(i)$, the adjacent northern location is denoted by $n(v)$ and the adjacent southern location is denoted by $s(v)$. $s(f(i))$ and $n(l(i))$ are defined similarly. By using auxiliary binary variables $\alpha$ and $\beta$, we consider the following feasible region $P_{sub}$ \hl{containing only subaisle cuts}:
\begin{alignat}{2} 
\label{sub1} \quad & \alpha_{tv} \geq \alpha_{ts(v)}, &  \qquad & \forall t\in \mathcal{T}, i\in [W_{sub}], v \in V_{sub}(i) \backslash \{n(l(i))\} \\
\label{sub2} \quad & x_{tn(v)v} \geq \alpha_{tv}, &  \qquad & \forall t\in \mathcal{T}, i\in [W_{sub}], v \in V_{sub}(i)\\
\label{sub3} \quad & \beta_{tv} \geq \beta_{tn(v)}, &  \qquad & \forall t\in \mathcal{T}, i\in [W_{sub}], v \in V_{sub}(i) \backslash \{s(f(i))\} \\
\label{sub4} \quad & x_{ts(v)v} \geq \beta_{tv}, &  \qquad & \forall t\in \mathcal{T}, i\in [W_{sub}], v \in V_{sub}(i)\\
\label{sub5} \quad & \alpha_{tv} + \beta_{tv} \geq z_{ot}, &  \qquad & \forall t\in \mathcal{T}, o\in O, v\in L_o \\
\label{sub6} \quad & \alpha_{tv} \in \{0,1\}, &  \qquad & \forall t\in \mathcal{T}, i\in [W_{sub}], v \in V_{sub}(i) \\
\label{sub7} \quad & \beta_{tv} \in \{0,1\}, &  \qquad & \forall t\in \mathcal{T}, i\in [W_{sub}], v \in V_{sub}(i) 
\end{alignat}

\hl{Those additional constraints work as follows}: we suppose \hl{that} there exists a feasible solution $(x^*,\alpha^*,\beta^*) \in P_{sub}$. For any picking location $v$ in subaisle $i$, if $\alpha^*_{tv}=1$ then there exists a \hl{straight} path connecting \hl{$f(i)$ and $v$} in walk $t$, and \hl{the same holds} when $\beta^*_{tv}=1$. \hl{If} walk $t$ pass through picking location $v$, \hl{then} there must exist a path connecting $v$ and the northern or southern artificial location. Therefore, constraints~(\ref{sub5}) are valid constraints. Remark that $\alpha_{tv} + \beta_{tv}$ may not be $2$ \hl{when} there exist two paths in walk $t$, one of them connects picking location $v$ and the northern artificial location, and others connect $v$ and the southern artificial location. \hl{In other words}, there is no surjection from $P_{basic}$ to $P_A=\{(x,y,z,\alpha,\beta):(x,y,z)\in P_{basic},(x,z,\alpha,\beta)\in P_{sub}\}$.

\hl{Subaisle cuts can be regarded as a class of connectivity constraints Therefore, those cuts can partially replace constraints~(\ref{bs4}). To illustrate the relationship between $P_{sub}$ and constraints~(\ref{bs4}), consider the following example.
}

\begin{example}
Consider the warehouse illustrated in Figure~\ref{feasol}. Suppose that there is an order $o$ with $L_o = \{v_{22},v_{31}\}$. We provide feasible solutions for various relaxations of $P_A$ in Figure~\ref{feasol}:
\begin{enumerate}
    \item When we remove all connectivity constraints, a feasible solution could only consist of several cycles. 
    \item When we remove constraints~(\ref{bs4}), $v_{22}$ and $v_{31}$ are forced to be connected with neighboring artificial locations.
    \item When all connectivity constraints are used, we are able to generate a Eulerian tour from a feasible solution.
\end{enumerate}
\end{example}

\begin{figure}
\centering
\subfigure[warehouse layout]{%
\resizebox*{5.5cm}{!}{\includegraphics{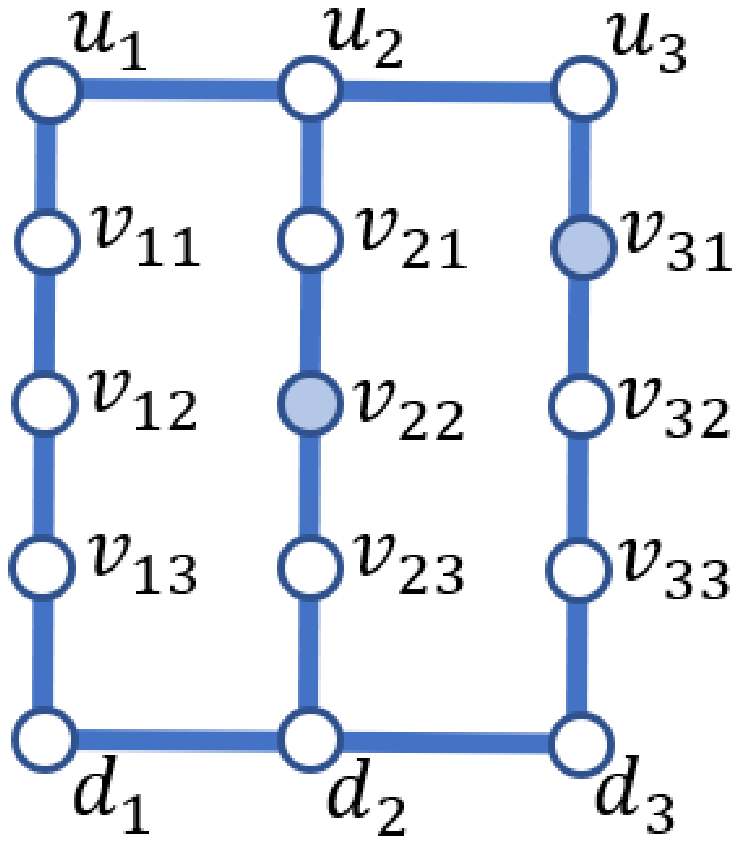}}}
\subfigure[$P_A$ without constraints~(\ref{bs4}) and $P_{sub}$]{%
\resizebox*{5.5cm}{!}{\includegraphics{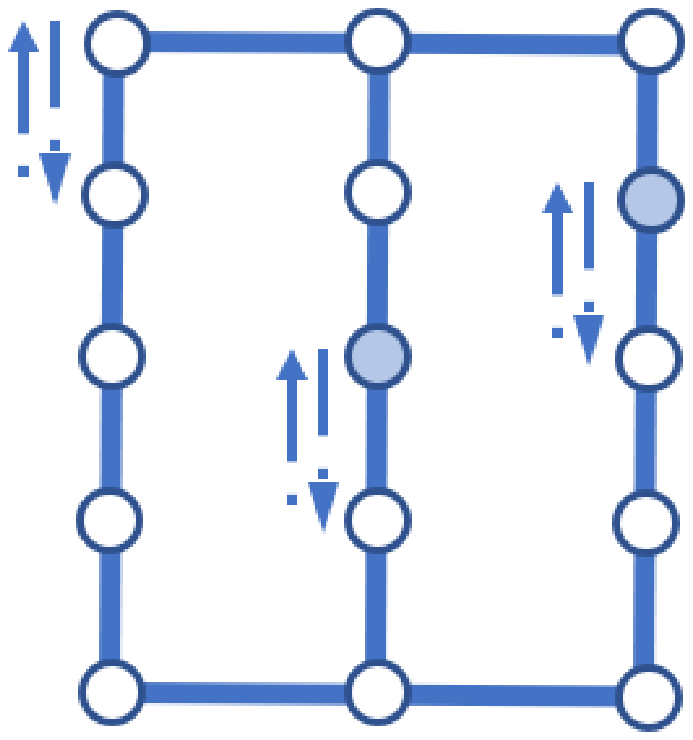}}}
\subfigure[$P_A$ without constraints~(\ref{bs4})]{%
\resizebox*{5.5cm}{!}{\includegraphics{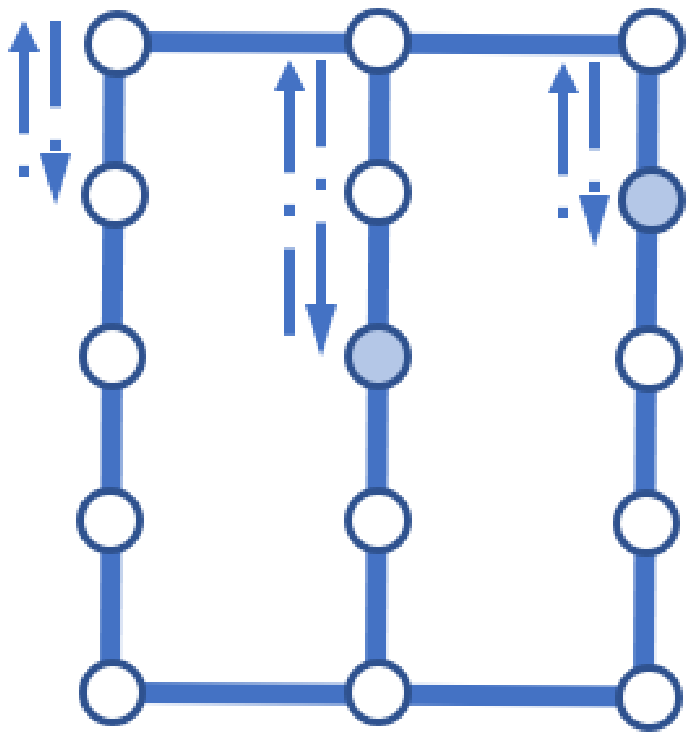}}}
\subfigure[$P_A$]{%
\resizebox*{5cm}{!}{\includegraphics{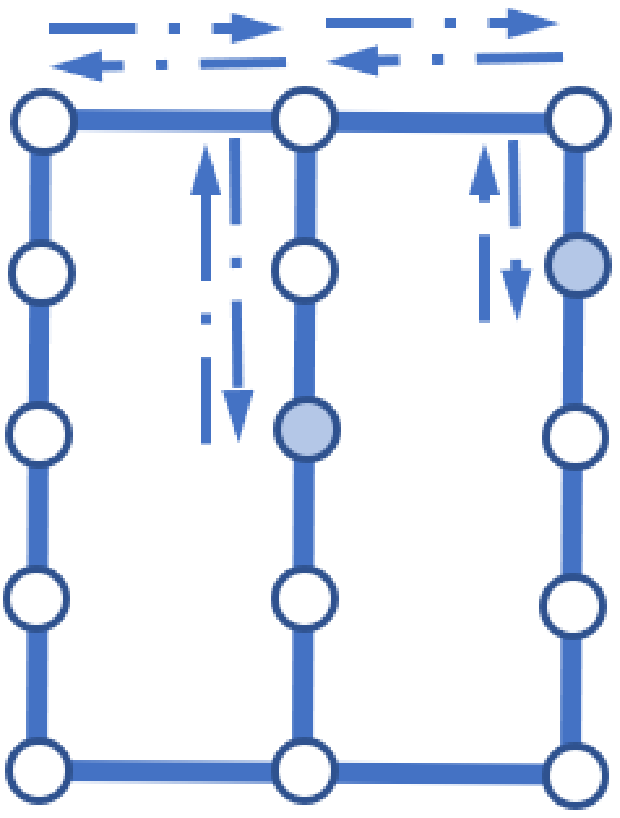}}}
\caption{Warehouse layout and feasible solutions for different formulations} \label{feasol}
\end{figure}

\hl{It is possible to observe that we only need to focus on the connectivity of the graph induced by artificial locations when subaisle cuts have been added to the formulation (Figure~\ref{feasol}(c)). In the remainder of this section, we reconstruct constraints~(\ref{bs4}) using this observation.} For any artificial location $v$, we let $Q_W(v)$ be the adjacent artificial location $u$ lying \hl{to} the west of $v$; we let $Q_E(v)$ be the adjacent artificial location $u$ lying \hl{to} the east of $v$. For any subaisle $i$, we let $Q_N(l(i))=f(i)$ and $Q_S(f(i))=l(i)$. We illustrate all \hl{of} these concepts in the reduced graph in Figure~\ref{subreduce}. \hl{Let $\tilde{E}'$ be the arc set of the reduced graph, that is, $\tilde{E}$ is a set of edges connecting neighboring artificial locations while ignoring picking locations within subaisles.} For each arc $(u,v)\in \tilde{E}'$, we introduce an auxiliary binary variable $\gamma_{tuv}$. For any node set $S\subset V_I$, let $\eta^{+}(S)=\{(u,v)\in \tilde{E}': u\in S,v\notin S\}$ and $\eta^{-}(S)=\{(u,v)\in \tilde{E}': u\notin S,v\in S\}$. \hl{We have} the following feasible region $P_g$:
\begin{alignat}{2} 
\label{impf1} \quad & \sum_{(s,v)\in \delta^{+}(s)}x_{tsv} \geq 1, & \qquad & \forall t\in \mathcal{T} \\
\label{impf2} \quad & \sum_{(u,v)\in \delta^{+}(u)}x_{tuv} \geq z_{ot}, &  \qquad & \forall t\in \mathcal{T}, o \in O, u \in L_o \\
\label{impf3} \quad & y_{tu} \geq x_{tuv}, & \qquad & \forall t\in \mathcal{T}, u\in V_I\backslash \{ s \}, (u,v) \in \delta^+(u) \\
\label{impf4} \quad & x_{tvQ_W(v)} = \gamma_{tvQ_W(v)}, & \qquad & \forall t\in \mathcal{T}, v\in V_I, (v,Q_W(v))\in \tilde{E}'\\
\label{impf5} \quad & x_{tvQ_E(v)} = \gamma_{tvQ_E(v)}, & \qquad & \forall t\in \mathcal{T}, v\in V_I, (v,Q_E(v))\in \tilde{E}' \\
\label{impf6.5} \quad & \alpha_{tn(l(i))} \geq \gamma_{tf(i)l(i)}, & \qquad & \forall t\in \mathcal{T}, i\in [W_{sub}]\\
\label{impf6} \quad & x_{tn(l(i))l(i)} \geq \gamma_{tf(i)l(i)}, & \qquad & \forall t\in \mathcal{T}, i\in [W_{sub}]\\
\label{impf7.5} \quad & \beta_{ts(f(i))} \geq
\gamma_{tl(i)f(i)}, & \qquad & \forall t \in \mathcal{T}, i \in [W_{sub}]\\
\label{impf7} \quad & x_{ts(f(i))f(i)} \geq \gamma_{tl(i)f(i)}, & \qquad & \forall t\in \mathcal{T}, i\in [W_{sub}]\\
\label{impf8} \quad & \sum_{(u,v) \in \eta^+(S)} \gamma_{tuv} \geq y_{tu_0}, & \qquad & \forall t\in \mathcal{T}, S \subset V_I\backslash \{ s \}, |S|\geq 2, u_0\in S \\
\label{impf9} \quad & \sum_{(v,u) \in \delta^+(v)} x_{tvu} = \sum_{(u,v) \in \delta^-(v)} x_{tuv}, & \qquad & \forall t\in \mathcal{T}, v\in V \\
\label{impf10} \quad & \sum_{t \in \mathcal{T}} z_{ot} = 1, & \qquad & \forall o \in O \\
\label{impf11} \quad & \sum_{o \in O} b_o z_{ot} \leq B, & \qquad & \forall t \in \mathcal{T} \\
\label{impf12} \qquad & x_{tuv} \in \{ 0,1 \}, & \qquad &  \forall t \in \mathcal{T}, (u,v) \in \tilde{E}\\
\label{impf13} \qquad & y_{tv} \in \{ 0,1 \}, & \qquad &  \forall t \in \mathcal{T}, v \in V\\
\label{impf14} \qquad & z_{ot} \in \{ 0,1 \}, & \qquad &  \forall t \in \mathcal{T}, o \in O\\
\label{impf15} \qquad & \gamma_{tuv} \in \{ 0,1 \}, & \qquad &  \forall t \in \mathcal{T}, (u,v) \in \tilde{E}'
\end{alignat}

\begin{figure}[htbp]
\centering
\resizebox*{8cm}{!}{\includegraphics{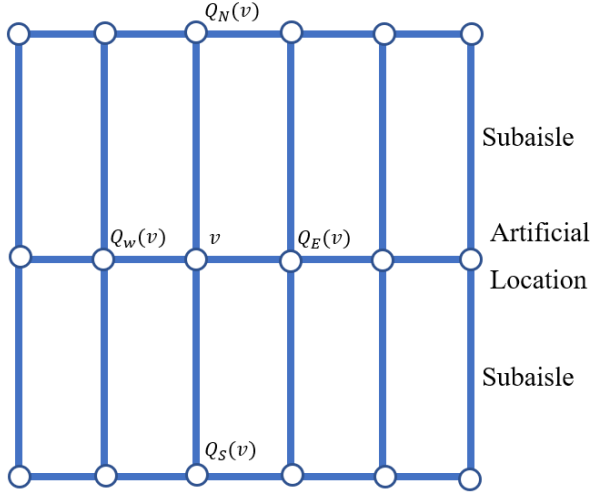}}
\hspace{5pt}
\caption{The related reduced graph} \label{subreduce}
\end{figure}

\hl{The auxiliary binary variable $\gamma_{tuv}$ tries to indicate whether $(u,v)\in \tilde{E}'$ is traversed by walk $t$. Consider a feasible solution $(x^*,y^*,z^*,\alpha^*,\beta^*,\gamma^*)\in P_g$. If $\gamma^*_{tuv}=1$, artificial locations $u$,$v$ are in the same connected component in walk $t$ according to constraints~(\ref{impf4})-(\ref{impf7}). If at most the subaisle between $u$ and $v$ is partially traversed in the direction $u\rightarrow v$ by walk $t$, then we have $\gamma^*_{tuv}=0$. Thus,} it suffices to add the \hl{connectivity} constraints~(\ref{impf8}) rather than constraints~(\ref{bs4}) to ensure that all artificial locations are in the same connected component. \hl{Let $f(x)=\sum_{t=1}^{T} \sum_{(u,v)\in \tilde{E}} d_{uv} x_{tuv}$ and $P_G=\{(x,y,z,\alpha,\beta,\gamma):(x,y,z,\gamma)\in P_g, (x,z,\alpha,\beta)\in P_{sub}\}$. As shown in the following theorem, it suffices to compute an optimal solution that satisfies $P_G$.}

\begin{theorem}\label{PG}
$min\{f(x):(x,y,z,\alpha,\beta,\gamma)\in P_G\} = 
min\{f(x):(x,y,z)\in P_{basic}\}$
\end{theorem}

\begin{proof}
For some $(x^*,y^*,z^*)\in P_{basic}$, $(\alpha^*,\beta^*,\gamma^*)$ are defined as follows:
\begin{enumerate}
    \item For any picking location $v$ within subaisle $i$, we set $\alpha_{tv}^*=1$ if $x_{tus(u)}^*=1$ for $u\in \{f(i),s(f(i)),s(s(f(i))),...,n(v)\}$; otherwise, we set $\alpha_{tv}^*=0$.
    \item For any picking location $v$ within subaisle $i$, we set $\beta_{tv}^*=1$ if $x_{tun(u)}^*=1$ for $u\in \{l(i),n(l(i)),n(n(l(i))),...,s(v)\}$; otherwise, we set $\beta_{tv}^*=0$.
    \item For any aritificial location $v$, we set $\gamma^*_{tvQ_{W}(v)} = x^*_{tvQ_{W}(v)}$ and $\gamma^*_{tvQ_{E}(v)} = x^*_{tvQ_{E}(v)}$. 
    We set $\gamma^*_{tvQ_{N}(v)}=1$ if $v=l(i)$ and $x_{tun(u)}^*=1$ for $u \in \{l(i),n(l(i)),...,s(f(i))\}$; otherwise, we set $\gamma^*_{tvQ_{N}(v)}=0$. We define $\gamma^*_{tvQ_{S}(v)}$ similarly.
\end{enumerate}
It is easy to check that  $(x^*,y^*,z^*,\alpha^*,\beta^*,\gamma^*)\in P_G$, therefore the following holds:
$$min\{f(x):(x,y,z,\alpha,\beta,\gamma)\in P_G\} \leq 
min\{f(x):(x,y,z)\in P_{basic}\}$$ 
Let us now consider an optimal solution $(x^*,y^*,z^*,\alpha^*,\beta^*,\gamma^*)\in P_{G}$. For picker $t\in \mathcal{T}$, the graph indicated by $x^*$ is denoted by $H$. The connected component of $H$ that contains the origin is denoted by $H_0$. We are able to generate a feasible picking tour from $H_0$ because any picking location $v$, which belongs to a location set $L_o$ with $z^*_{ot}=1$, also belongs to the vertex set of $H_0$. This gives us the following result.
$$min\{f(x):(x,y,z,\alpha,\beta,\gamma)\in P_G\} \geq
min\{f(x):(x,y,z)\in P_{basic}\}$$
\end{proof}

\hl{Similar to $P_{basic}$, the improved formulation $P_G$ involves exponentially many constraints to enforce connectivity. The connectivity constraints should be added} to the model via a separation procedure when using a branch-and-cut algorithm. For any candidate integral solution in the branch-and-cut tree, we verify via a depth-first search whether the graph is connected and \hl{we} add \hl{connectivity} constraints for every connected component except the one containing the origin when the graph is not connected. \hl{Note that because} a subset of constraints~(\ref{bs4}) \hl{is} actually dominated by a single constraint~(\ref{impf8}) according to constraints~(\ref{impf4})-(\ref{impf7}), constraints~(\ref{impf8}) \hl{seem to} be more suitable for a branch-and-cut procedure.

We also provide a compact formulation of this problem, i.e., a formulation with a polynomial number of variables and constraints. The main trick is introducing an auxiliary multicommodity flow problem. Assume that there is a \hl{salesman} in each selected artificial location. The \hl{salesmen} must determine feasible \hl{paths} to the origin. Once all salesmen can get to the origin, the origin and all selected artificial locations are in the same connected component. We introduce flow variables $\sigma^{v_0}_{tuv}$ to indicate the amount of commodity from artificial location $v_0$ passing through arc $(u,v)\in \tilde{E}'$ in walk $t$. $P_f$ is given by constraints~(\ref{impf1})-(\ref{impf7}), (\ref{impf9})-(\ref{impf15}) and the following:
\begin{alignat}{2} 
\label{impcf1} \quad & \sum_{(v_0,v)\in \eta^{+}(v_0)}\sigma^{v_0}_{tv_0v} - \sum_{(v,v_0)\in \eta^{-}(v_0)}\sigma^{v_0}_{tvv_0} = y_{tv_0}, &  \qquad & \forall t\in \mathcal{T}, v_0\in V_I \\
\label{impcf2} \quad & \sum_{(u,v)\in \eta^{+}(u)}\sigma^{v_0}_{tuv} -  \sum_{(v,u)\in \eta^{-}(u)}\sigma^{v_0}_{tvu} = 0, &  \qquad & \forall t\in \mathcal{T}, v_0\in V_I, u\in V_I\backslash \{s,v_0 \} \\
\label{impcf3} \quad & \sum_{(s,v)\in \eta^{+}(u)}\sigma^{v_0}_{tsv} -  \sum_{(v,s)\in \eta^{-}(s)}\sigma^{v_0}_{tvs} = -y_{tv_0}, &  \qquad & \forall t\in \mathcal{T}, v_0\in V_I \\
\label{impcf4} \quad & 0 \leq \sigma^{v_0}_{tuv} \leq \gamma_{tuv}, &  \qquad & \forall t\in \mathcal{T}, v_0\in V_I, (u,v) \in \tilde{E}'
\end{alignat}

\hl{Let $P_F=\{(x,y,z,\alpha,\beta,\gamma,\sigma):(x,y,z,\gamma,\sigma)\in P_f, (x,z,\alpha,\beta)\in P_{sub}\}$. We immediately obtain the following.
}
\begin{theorem}\label{PF}
$min\{f(x):(x,y,z,\alpha,\beta,\gamma,\sigma)\in P_F\}=min\{f(x):(x,y,z)\in P_{basic}\}$
\end{theorem}
\begin{proof}
For some $(x^*,y^*,z^*)\in P_{basic}$, $(\alpha^*,\beta^*,\gamma^*)$ are defined as in the proof of Theorem~\ref{PG} and $\sigma^*$ is defined as follows.

For picker $t\in \mathcal{T}$, the graph indicated by $x^*$ is denoted by $H$. For any artificial location $v_0$ with $y^*_{tv_0}=1$, we consider an arbitrary path in $H$ that goes from $v_0$ to the origin (because $(x^*,y^*,z^*)\in P_{basic}$, such a path exists). We set $\sigma^{v_0*}_{tuv}=1$ if $(u,v)$ is traversed by this path; otherwise, we set $\sigma^{v_0*}_{tuv}=0$. 

It is easy to check that  $(x^*,y^*,z^*,\alpha^*,\beta^*,\gamma^*,\sigma^*)\in P_F$, and therefore the following holds:
$$min\{f(x):(x,y,z,\alpha,\beta,\gamma,\sigma)\in P_F\} \leq 
min\{f(x):(x,y,z)\in P_{basic}\}$$ 

The rest of the proof is similar to that of Theorem~\ref{PG}.
\end{proof}

\hl{We finish this section by showing} the relationship between the LP relaxations of $P_G$ and $P_F$. Formally, let $P_{LP}$ denote the LP relaxation of $P$, \hl{let $proj_{x}(P)$ denote the projection of $P$ onto the $x-space$} and we prove the following theorem:
\begin{theorem}\label{LPrelax}
$proj_{(x,y,z,\alpha,\beta,\gamma)}((P_F)_{LP})=(P_G)_{LP}$
\end{theorem}
\begin{proof}\label{LPrelaxprf}
According to [Theorem 2. \citep{LETCHFORD201383}], there exists a feasible flow $\sigma^{u_0}_{t}$ satisfying constraints~(\ref{impcf1})-(\ref{impcf4}) if and only if
$$
\sum_{(u,v)\in \eta^+(S)}\gamma_{tuv} \geq y_{tu_0}\qquad \forall S\in \{S: u_0\in S, S\subset V_I\backslash \{ s \} \}
$$
Because there should be a feasible flow $\sigma^{u_0}_{t}$ for any $u_0\in V_I\backslash \{s\}$ according to our construction, the proof is ended.
\end{proof}

\section{Additional constraints}\label{additional_cons}

In this section, we introduce two types of additional constraints\hl{, called strengthened connectivity constraints and single traversing constraints. Strengthened connectivity constraints link the routing decisions and the batching decisions where classical connectivity constraints cannot take batching decisions into account. Single traversing constraints mainly focus on the problem property of a rectangular warehouse and can cut off some non-optimal solutions.}

\subsection{Strengthened \hl{connectivity} constraints \hl{and basic cuts}}\label{subsec-scc}

\hl{Considering the relationship between routing decisions and batching decisions, it would be interesting to investigate the constraints that jointly deal with both decisions.}

\hl{We first introduce the concept of the strengthened connectivity constraints.} Let us have a look at \hl{connectivity} constraints~(\ref{bs4}): a constraint of type~(\ref{bs4}) can be generated by fixing a proper vertex set $S$. For any picking location $u_0\in S$, we know that $y_{tu_0}$ is actually dominated by any $z_{ot}$ satisfying $u_0\in L_o$ according to constraints~(\ref{bs2})-(\ref{bs3}). Therefore, we have the following strengthened \hl{connectivity} constraints:
\begin{alignat}{2} 
\label{stenc} \quad & \sum_{(u,v)\in \delta^+(S)}x_{tuv} \geq z_{ot}, &  \qquad & \forall t\in \mathcal{T}, S\subset V\backslash\{s\}, |S|\geq 2, o\in \{o: S\cap L_o\neq \emptyset\}
\end{alignat}

\hl{Notice that there are an exponential number of connectivity constraints~(\ref{bs4}). Thus, constraints~(\ref{stenc}) cannot be added directly to a formulation. To make use of these constraints, we intend to find a family $\mathcal{S}$ of a polynomial number of vertex sets $S$. All constraints induced by $\mathcal{S}$ will be added directly to a formulation.} A carefully selected $\mathcal{S}$ leads to powerful constraints. For instance, \hl{if $S\in \mathcal{S}$ is composed of locations on the right side of some aisle (as shown in Figure~\ref{strencc}), constraints~(\ref{stenc}) are reduced to aisle cuts proposed by \citep{VALLE2017817}.}
\begin{figure}
\centering
\subfigure[Aisle cut]{%
\resizebox*{6cm}{!}{\includegraphics{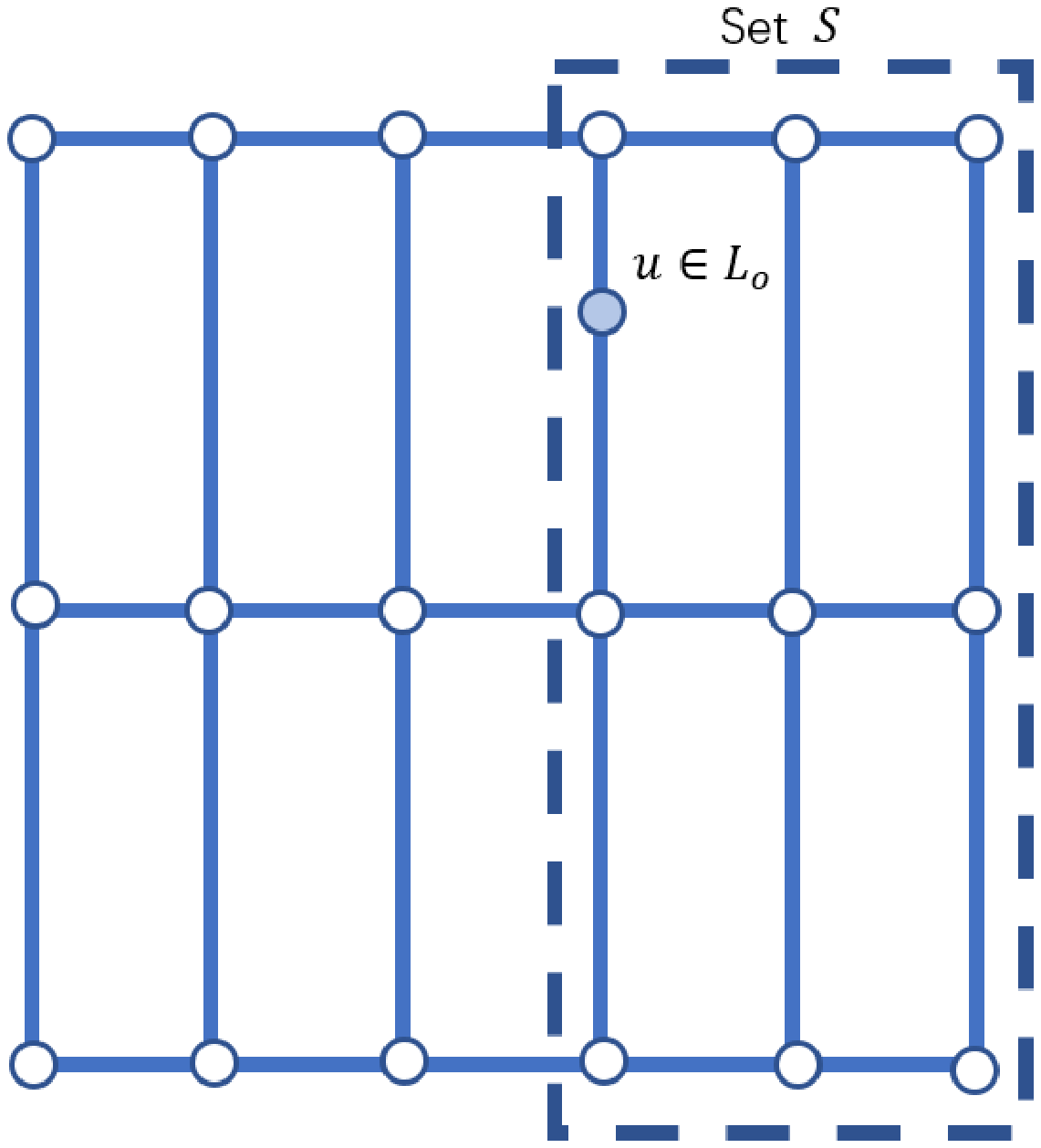}}}
\subfigure[Basic cut]{%
\resizebox*{6cm}{!}{\includegraphics{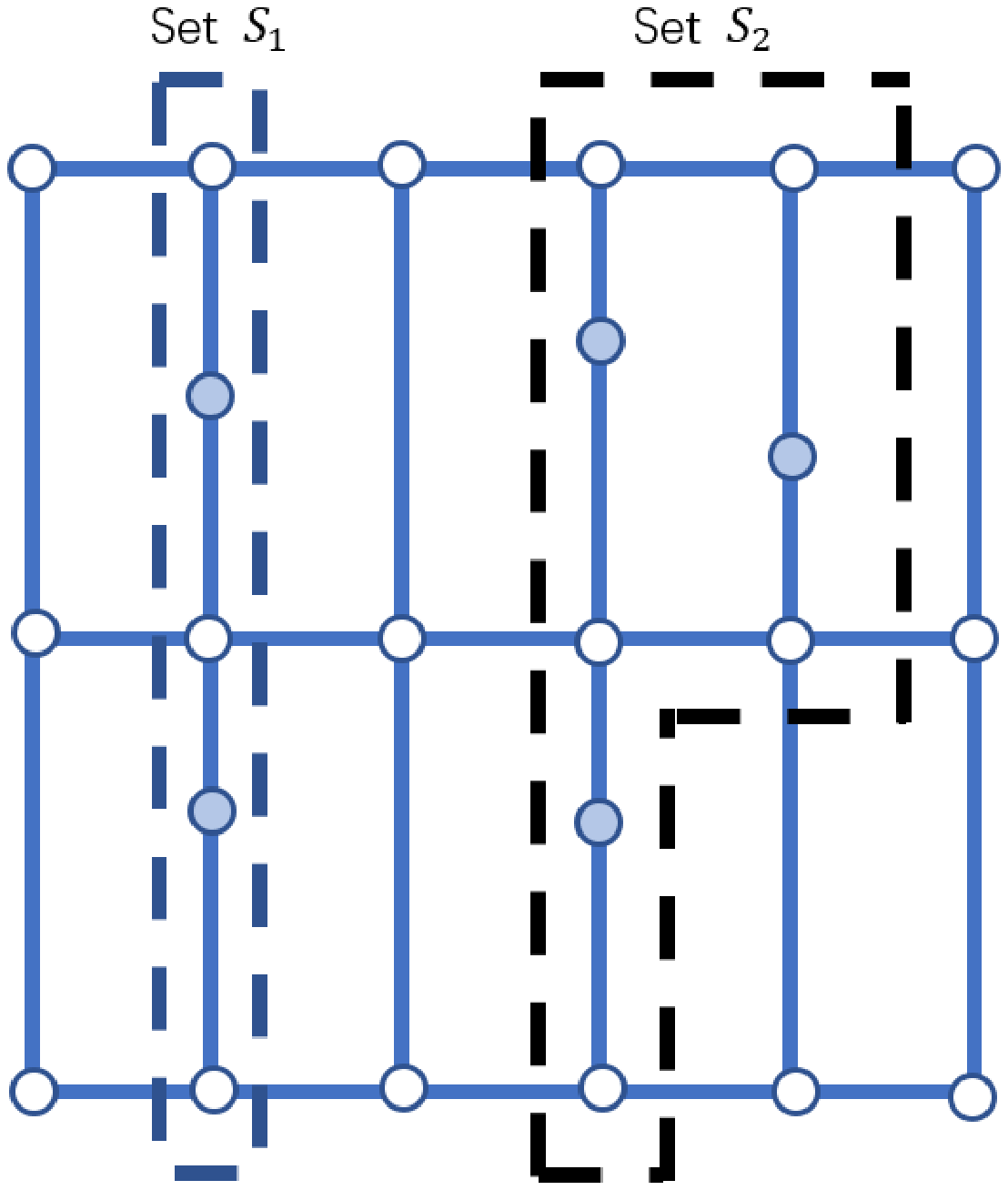}}}
\caption{Strengthened \hl{connectivity} constraints} \label{strencc}
\end{figure}

Here, we present another type of strengthened \hl{connectivity} constraints, which \hl{we call} basic cuts. \hl{$\mathcal{S}$ can be obtained by the following operations. For any order $o$, let $\tilde{E}_0$ denote a set of $e \in \tilde{E}'$ (which represents a subaisle) that contains at least one picking location $u\in L_o$. Let $\tilde{V}_0$ be a set of vertices associated with $\tilde{E}_0$. Then, any $e \in \tilde{E}'$ with both ends in $\tilde{V}_0$ will be added to $\tilde{E}_0$.} By executing a depth-first search algorithm, we can discover all connected components of the graph \hl{$(\tilde{V}_0,\tilde{E}_0)$. Figure~\ref{strencc} shows an example of the two connected components induced by an order. Let $S$ be a vertex set of a connected component that does not contain the origin. All picking locations within the subaisle that has both ends in $S$ will then be added to $S$. Then $S$ will be added to $\mathcal{S}$. We call the strengthened connectivity constraints induced by $\mathcal{S}$ basic cuts.}

\subsection{Single traversing constraints}\label{subsec-stc}

\hl{In this subsection, we derive additional constraints by investigating the graph property of the warehouse. We first restrict our attention to a single-block warehouse. By 'traversing', we mean going from the north artificial vertex to the south artificial vertex or
vice-versa.}

\begin{theorem}\label{singletrip} \hl{For a single-block warehouse,} each subaisle will be traversed at most once by any optimal walk. 
\end{theorem}
\begin{corollary}\label{singletrip_cons} (Single traversing constraints) \hl{For a single-block warehouse}, no optimal solution will be cut off by the following constraints:
\begin{alignat}{2} 
\label{sitr} \quad & \alpha_{tv} + \beta_{tv} \leq 1, & \qquad & \forall t\in \mathcal{T}, o\in O, v\in L_o 
\end{alignat}
\end{corollary}

We remark that a subset of feasible solutions may \hl{not satisfy Theorem~\ref{singletrip}}; therefore, \hl{the inequalities it yields might not be} strictly valid inequalities. Furthermore, \hl{Theorem~\ref{singletrip}} can induce different constraints other than~(\ref{sitr}) and it is still unclear whether~(\ref{sitr}) is the most proper formulation. 

\hl{Corollary~\ref{singletrip_cons} can be obtained as a direct consequence of Theorem~\ref{singletrip}, and thus, it suffices to prove Theorem~\ref{singletrip}.} For a single-block warehouse, let the artificial locations in the first cross-aisle be $u_1,u_2,...,u_n$, and the artificial locations in the second cross-aisle \hl{be} $d_1,d_2,...,d_n$. Let picking locations in the path north to south in aisle $i$ be $v_{i1},v_{i2},...,v_{is_i}$. We illustrate all \hl{of} these concepts in Figure~\ref{single_cut}. \hl{A} Eulerian tour can be specified using a sequence of vertices, such as 
\begin{alignat}{2} 
r = (u_{1}v_{11}v_{12}...d_{1}d_{2}...u_{2}u_{1}) \notag
\end{alignat}

\begin{figure}
\centering
\resizebox*{8.5cm}{!}{\includegraphics{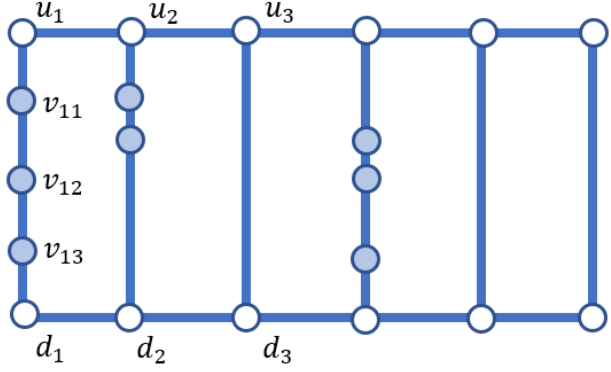}}
\hspace{5pt}
\caption{notations that we used in subsection~\ref{subsec-stc}} \label{single_cut}
\end{figure}

Let $V(r)$ denotes the set \hl{of} the vertices in $r$. We say that Eulerian tour $r'$ is better than Eulerian tour $r$ if $r,r'$ satisfy
\begin{enumerate}
  \item $\{V(r)\cap V_L\}$ \colorbox{yellow}{$\subseteq$}
  $\{V(r')\cap V_L\}$
  \item $r'$ is shorter than $r$
\end{enumerate}

Now, we introduce three lemmas to prove Theorem~\ref{singletrip}.
\begin{lemma}\label{sig1}
For any tour $r$, if there exists an aisle $i$ such that tour $r$ contains subpath $(u_{i}v_{i1}v_{i2}...v_{is_i}d_{i}v_{is_i}v_{i(s_{i}-1)}...v_{i1}u_{i})$, then there exists a tour $r'$ better than $r$.
\end{lemma}
\begin{proof}
We simply assume that $r=(r_1r_2...r_pu_{i}v_{i1}v_{i2}...v_{is_i}d_{i}v_{is_i}v_{i(s_{i}-1)}...v_{i1}u_{i}r_q...r_R)$. Then $r'=(r_1r_2...r_pu_{i}v_{i1}v_{i2}...v_{i(s_{i}-1)}v_{is_i}v_{i(s_{i}-1)}...v_{i1}u_{i}r_q...r_R)$ is a better tour.
\end{proof}

\hl{The reader may find a different application of Lemma~\ref{sig1} in \citep{VALLE2017817}, in which it also induced the artificial vertex reversal cuts.}

\begin{lemma}\label{sig2}
For any tour $r$, if there exists an aisle $i$ such that tour $r$ contains subpath $(u_{i-1}u_{i}v_{i1}v_{i2}...v_{is_i}d_{i}d_{i+1})$ and $(d_{i}v_{is_i}v_{i(s_i-1)}...v_{i_1}u_{i})$, then there exists a tour $r'$ better than $r$.
\end{lemma}
\begin{proof}
We simply assume that $r=(r_1r_2...r_pu_{t}u_{t+1}...u_{i-1}u_{i}v_{i1}v_{i2}...v_{is_i}d_{i}d_{i+1}...d_kv_{ks_k}r_q...r_R)$. Then $r'=(r_1r_2...r_pu_{t}u_{t+1}...u_{i-1}u_{i}...u_{k}v_{k1}v_{k2}...v_{ks_k}r_q...r_R)$ is a better tour.
\end{proof}

\begin{lemma}\label{sig3}
For any tour $r$, if there exists an aisle $i$ such that tour $r$ contains subpath $(u_{i-1}u_{i}v_{i1}v_{i2}...v_{is_i}d_{i}d_{i-1})$ and $(d_{i}v_{is_i}v_{i(s_i-1)}...v_{i_1}u_{i})$, then there exists a tour $r'$ better than $r$.
\end{lemma}
\begin{proof}
We simply assume that $r=(r_1r_2...r_pu_{t}u_{t+1}...u_{i-1}u_{i}v_{i1}v_{i2}...v_{is_i}d_{i}d_{i-1}...d_kv_{ks_k}r_q...r_R)$ with $k\geq t$. Then $r'=(r_1r_2...r_pu_{t}u_{t+1}...u_{i-1}u_{i}...u_{k}v_{k1}v_{k2}...v_{ks_k}r_q...r_R)$ is a better tour.
\end{proof}

Now, we can give the proof for Theorem~\ref{singletrip}.
\begin{proof}
We assume by contradiction that there exists an optimal tour $r=(r_1r_2...r_pu_{i-1}u_iv_{i1}...v_{is_i}d_ir_q...r_R)$ and $r$ contains subpath $(d_{i}v_{is_i}v_{i(s_i-1)}...v_{i_1}u_{i})$. Then
\begin{enumerate}
  \item If $(r_qr_{q+1}...r_{q+s_i})=(v_{is_i}v_{i(s_i-1)}...v_{i_1}u_{i})$, lemma~\ref{sig1} provides a better tour.
  \item If $r_q=d_{i+1}$, lemma~\ref{sig2} provides a better tour.
  \item If $r_q=d_{i-1}$, lemma~\ref{sig3} provides a better tour.
\end{enumerate}
Thus, no such optimal tour exists.
\end{proof}

\hl{In the remainder of this subsection, we extend the result for a 2-block warehouse. Let $1$ be the index of the topmost left subaisle, and we have the following theorem.
}

\begin{theorem}\label{twotrip} \hl{For a 2-block warehouse, there exists an optimal walk that traverses any subaisle $i\in \{2,3,...,W_{sub}\}$ at most once.}
\end{theorem}

\begin{proof}
Let $r$ be an optimal walk. We can partition this tour into several parts $r_1,r_2,...,r_R,r_0$ and each part is located in the first block or the second block (see Figure~\ref{revcon2fig}). As shown in the proof of Theorem~\ref{singletrip}, a subaisle cannot be traversed twice by any subtour $r_k,k\in \{0,1,...,R\}$. Furthermore, there always exists an optimal solution such that a subaisle will not be traversed by different subtours $r_k,r_t$ where $k,t\in \{1,2,...,R\}$. The construction of such an optimal solution is similar to the proofs of Lemma~\ref{sig2}-\ref{sig3}. Thus, there exists an optimal walk $r$ such that any subaisle is traversed at most once by $r'=(r_1,r_2,...,r_R)$. If a picker needs to pick some products in $r_0$, then any subaisle is traversed at most once by $r$; otherwise, the picker needs to return to the origin immediately, and we can assume that the picker first goes to the south artificial vertex of subaisle $1$ and then goes to the origin. This completes the proof.

\begin{figure}
\centering
\resizebox*{8.5cm}{!}{\includegraphics{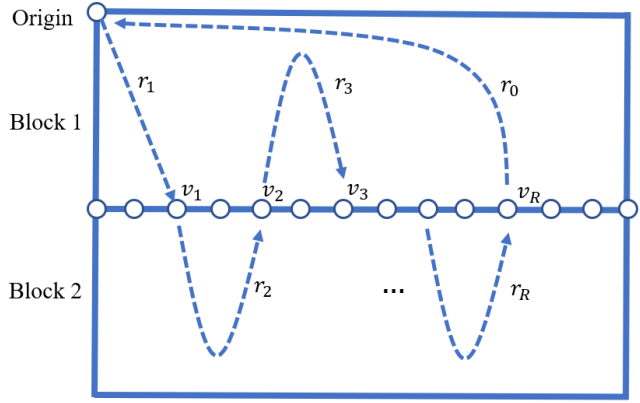}}
\hspace{5pt}
\caption{The partition of the walk $r$} \label{revcon2fig}
\end{figure}

\end{proof}

\begin{corollary}\label{twotrip_cons} \hl{(Single traversing constraints) For a 2-block warehouse, there exists an optimal solution satisfying the following constraints:}
\begin{alignat}{2} 
\label{sitr2} \quad & \alpha_{tv} + \beta_{tv} \leq 1, & \qquad & \forall t\in \mathcal{T}, o\in O, v\in L_o \backslash V_{sub}(1) 
\end{alignat}
\end{corollary}
\begin{proof}
This follows immediately from Theorem~\ref{twotrip}.
\end{proof}

\section{No-reversal case}\label{no-revcase}

\hl{In this section, we introduce the no-reversal case of the JOBPRP and present new formulations that are less influenced by symmetric solutions.
}

We note that it would be impractical to search for an optimal solution when there are a vast number of orders. By \hl{assuming that routing is conducted in a no-reversal fashion}, \citep{VALLE2017817} obtained an efficiently solvable formulation. The no-reversal constraints are described as follows.
\begin{alignat}{2} 
\label{norev1} \quad & x_{tn(v)v} = x_{tv}, &  \qquad & \forall t\in \mathcal{T}, i\in [W_{sub}], v \in V_{sub}(i) \cup \{l(i)\} \\
\label{norev2} \quad & x_{ts(v)v} = x_{tv}, &  \qquad & \forall t\in \mathcal{T}, i\in [W_{sub}], v \in V_{sub}(i) \cup \{f(i)\} 
\end{alignat}

In the case of no-reversal constraints, a picker is not allowed to reverse in a subaisle. Many feasible solutions, including some or all optimal solutions, will be cut off by these constraints. \hl{Note that the no-reversal JOBPRP can also} be regarded as an approximate problem, \hl{and by solving this problem, we can quickly make order batching decisions. Then each picker can be rerouted by applying other methods such as a classical TSP model.}

\hl{Although these constraints can greatly simplify the decision problem, } this problem still suffers severely from the presence of symmetry. A large amount of equivalence of solutions induced by symmetry routes might confound the branch-and-bound \hl{(or branch-and-cut)} process, as illustrated in Figure~\ref{symrt}. \hl{In other words, there is also a substantial room for the further improvement of the no-reversal formulation.}

\begin{figure}
\centering
\subfigure{%
\resizebox*{6cm}{!}{\includegraphics{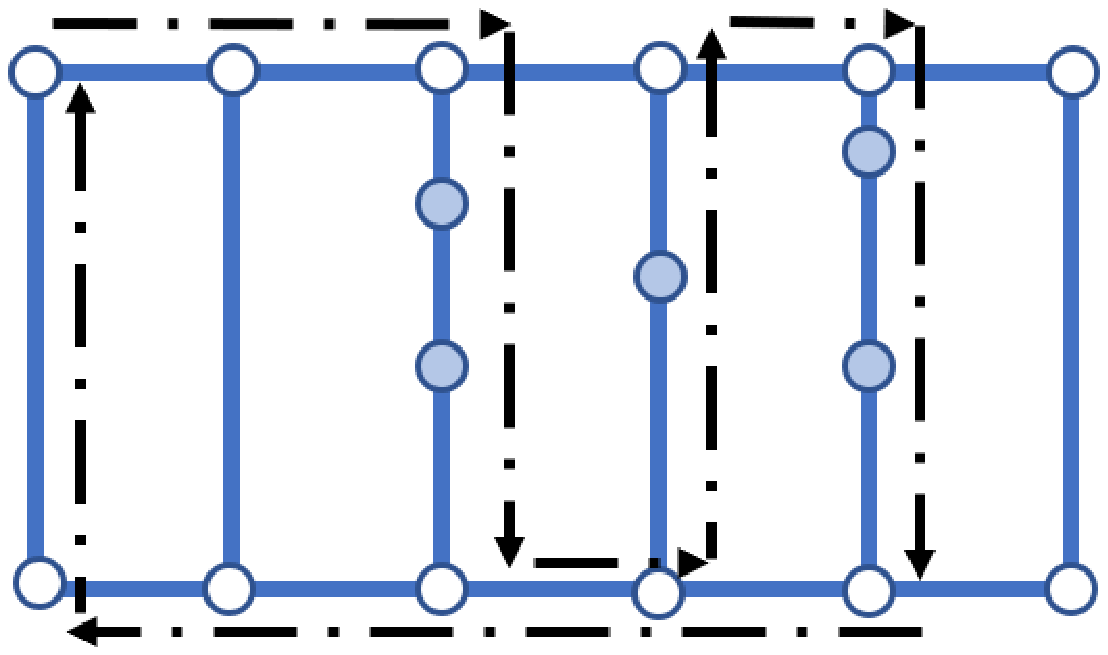}}}
\subfigure{%
\resizebox*{6cm}{!}{\includegraphics{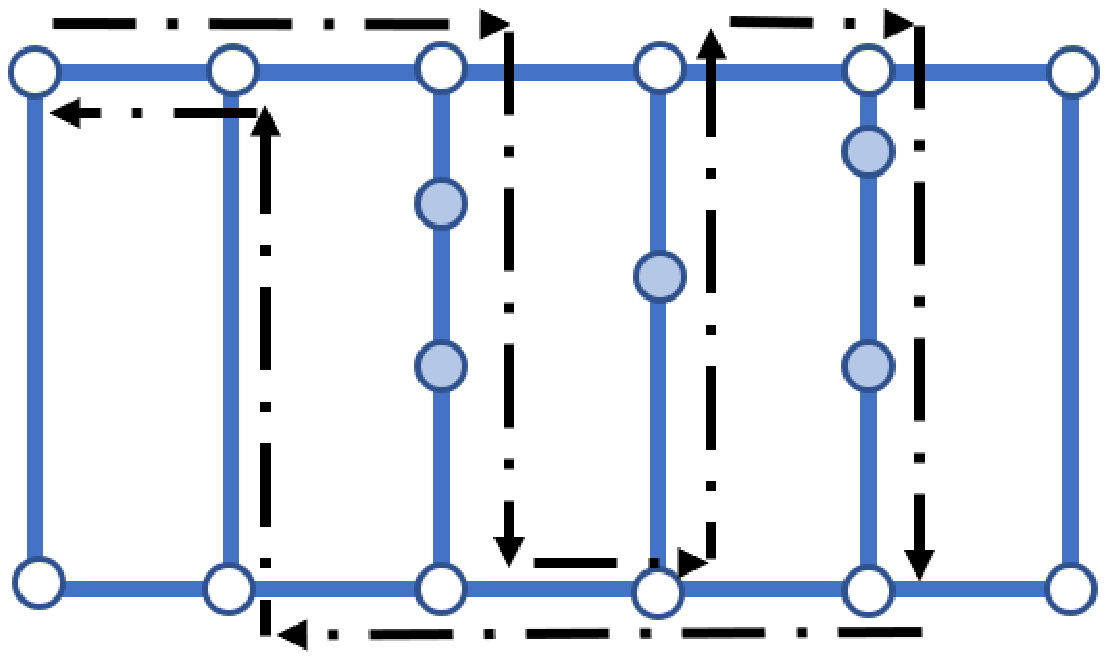}}}
\caption{Symmetry routes} \label{symrt}
\end{figure}

\subsection{TSP formulation for a single-block warehouse}

A single-block warehouse is first considered. We start with an undirected graph $G^+=(V_I,E^+)$, which is constructed as follows: Recall that $e=[u,v]$ denote the unordered pair of location $u$ and location $v$. $E^+$ is defined as $E_1^+\cup E_2^+$ where $E_1^+=\{[u,v]:(u,v)\in \tilde{E}'\}$ and $E_2^+=\{[s,v]:v\in V_I\backslash \{s\} \}$. For any subaisle $i$, let $e(i)$ denote the edge $[f(i),l(i)]$. The resulting feasible region \hl{$P^{1}_{U}$} then consists of the following constraints.
\begin{alignat}{2} 
\label{tspo0} \quad & x_{t[s,l(1)]} + x_{t[s,f(2)]} \geq 1, &  \qquad & \forall t\in \mathcal{T} \\
\label{tspo1} \quad & \sum_{[s,v]\in \delta(s)}x_{t[s,v]} + \tilde{x}_t = 2, &  \qquad & \forall t\in \mathcal{T} \\
\label{tspo2} \quad & x_{t[u,v]} \geq z_{ot}, &  \qquad & \forall t\in \mathcal{T}, i\in [W_{sub}], [u,v] = e(i), o\in \{ o: V_{sub}(i) \cap L_o \neq \emptyset\} \\
\label{tspo3} \quad & \sum_{[l(1),v] \in \delta(l(1))} x_{t[l(1),v]} + \tilde{x}_t = 2 y_{tu}, & \qquad & \forall t\in \mathcal{T} \\
\label{tspo4} \quad & \sum_{[u,v] \in \delta(u)} x_{t[u,v]} = 2 y_{tu}, & \qquad & \forall t\in \mathcal{T},  u \in V_I\backslash \{ s,l(1) \} \\
\label{tspo5} \quad & \sum_{[u,v] \in \delta(S)} x_{t[u,v]} \geq 2 y_{tu_0}, & \qquad & \forall t\in \mathcal{T},  S \subset V_I\backslash \{ s \}, |S|\geq 2, u_0\in S \\
\label{tspo6} \quad & \sum_{t \in \mathcal{T}} z_{ot} = 1, & \qquad & \forall o \in O \\
\label{tspo7} \quad & \sum_{o \in O} b_o z_{ot} \leq B, & \qquad & \forall t \in \mathcal{T} \\
\label{tspo8} \qquad & x_{t[u,v]} \in \{ 0,1 \}, & \qquad &  \forall t \in \mathcal{T}, [u,v] \in E^+\\
\label{tspo9} \qquad & \tilde{x}_{t} \in \{ 0,1 \}, & \qquad &  \forall t \in \mathcal{T} \\
\label{tspo10} \qquad & y_{tv} \in \{ 0,1 \}, & \qquad &  \forall t \in \mathcal{T}, v \in V_I\\
\label{tspo11} \qquad & z_{ot} \in \{ 0,1 \}, & \qquad &  \forall t \in \mathcal{T}, o \in O
\end{alignat}
Here, $x_{t[u,v]}$ is a binary variable, which takes the value $1$ if and only if edge $[u,v]$ \hl{is} traversed by walk $t$. We let binary variable $\tilde{x}$ indicate whether a parallel edge between $f(1)$ and $s(1)$ should be added (note that the origin $s=f(1)$). We actually formulate the standard TSP \hl{on $E^+$ (possibly with a parallel edge of $e(1)$)} using constraints~(\ref{tspo0})-(\ref{tspo5}). We illustrate how our formulation works in Figure~\ref{tspexample}. 

\hl{The TSP-based formulation is highly dependent on the construction of the auxiliary graph. Once we define the undirected graphs for other warehouses, we are able to propose extended formulations. In the next subsection, we discuss the auxiliary graph for a 2-block warehouse.}

\begin{figure}
\centering
\resizebox*{7cm}{!}{\includegraphics{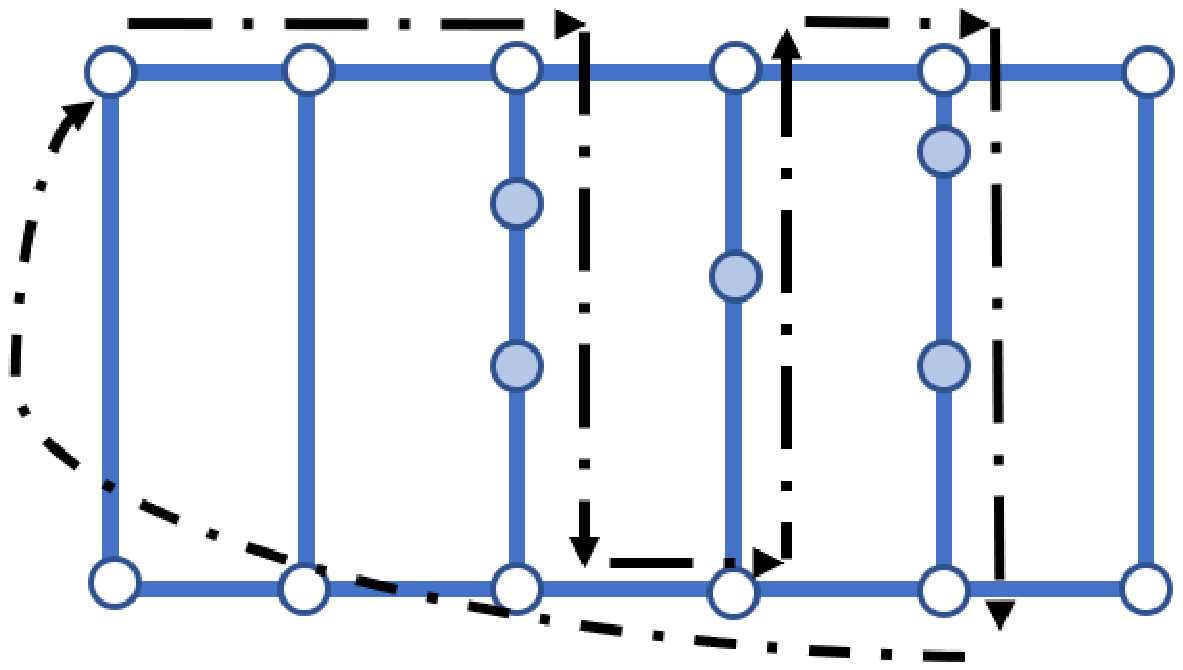}}
\hspace{5pt}
\caption{The only path back to the origin} \label{tspexample}
\end{figure}

\subsection{TSP formulation for a 2-block warehouse}\label{tsp2block}

\hl{We} should propose a more sophisticated auxiliary graph which can \hl{induce} a TSP formulation that contains an optimal solution. An optimal route in a 2-block warehouse may be different from an optimal route in a single-block warehouse such that:
\begin{enumerate}
  \item The picker may require to traverse a subaisle twice to enter the other block.
  \item The picker may require to traverse some edge in the second cross-aisle twice before retrieving all products from storage (after \hl{which} the picker should go back to the origin).
\end{enumerate}

\hl{From the above observations, we should add more parallel edges to the graph to preserve optimal solutions. In fact, we are describing the problem property when adding parallel edges. Once the auxiliary graph maintains a subset of optimal solutions, we obtain a TSP-based formulation.}

Now, we show how to construct the auxiliary graph. Let $W$ be a set of artificial locations in the second cross-aisle. We construct a copy $v'$ for any $v\in W$ and let the set of copies be $W'$. \hl{Let $E_1^+$ be the set of edges connecting neighboring vertex. $E_2^+$ and $E_3^+$ are defined as follows.}
\begin{enumerate}
  \item $E_2^+=\{[v',Q_N(v)]:v\in W\} \cup \{[v,Q_S(v')]:v\in W\}$
  \item $E_3^+=\{[s,v]:v\in \{ V \cup W' \} \backslash \{ s\} \}$
\end{enumerate}

\hl{$E_1^+$ enables the picker to reverse direction in the second cross-aisle; $E_2^+$ enables the picker to traverse a subaisle twice; $E_3^+$ enables the picker to return to the origin after picking is completed.}
We compare the two auxiliary graphs in this section, as shown in Figure~\ref{tspcon}.
\begin{figure}
\centering
\subfigure[$E^+$ for a single-block warehouse]{%
\resizebox*{6cm}{!}{\includegraphics{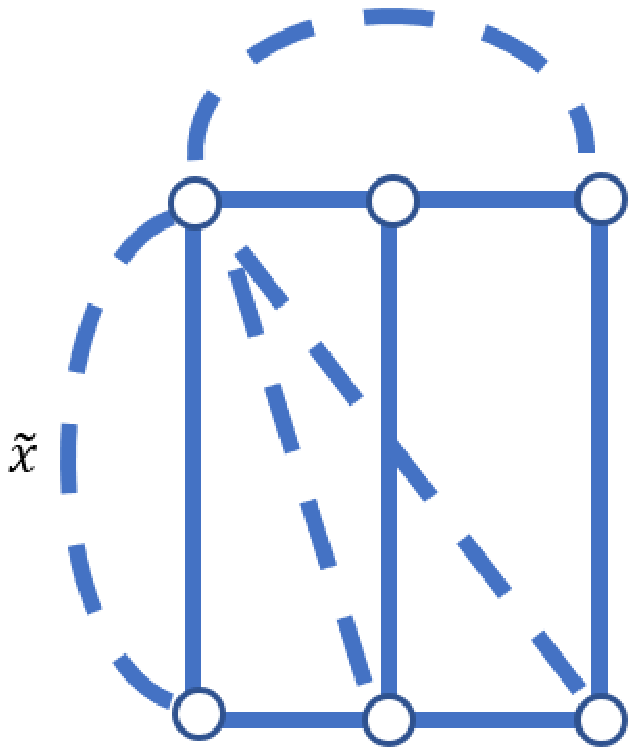}}}
\subfigure[$\cup_{i=1}^3 E_i^+$ for a 2-block warehouse]{%
\resizebox*{5cm}{!}{\includegraphics{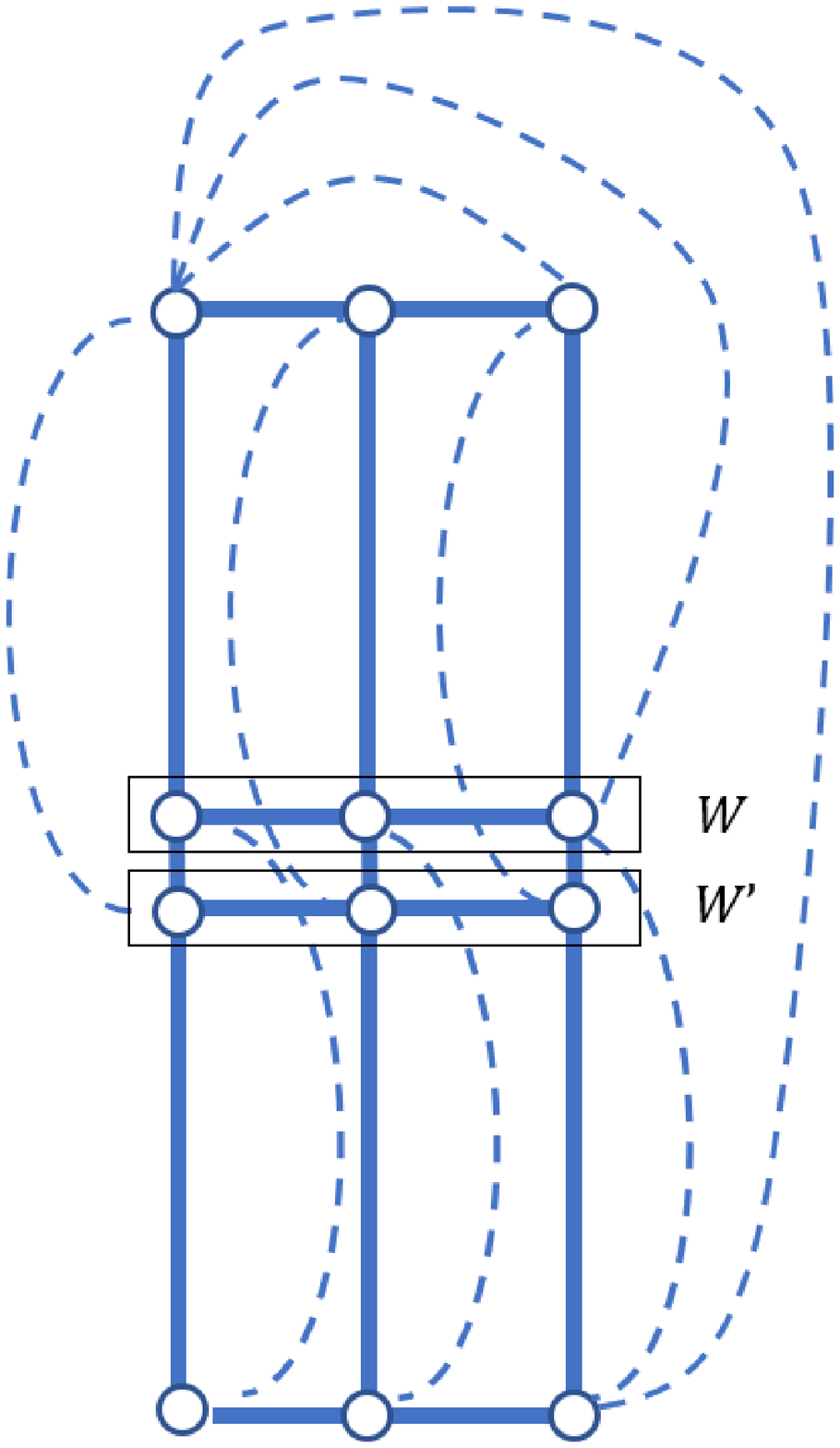}}}
\caption{Auxiliary graphs} \label{tspcon}
\end{figure}

\hl{For each picker $t\in \mathcal{T}$, let $x_{t[u,v]}\in \{0,1\}$ be an indicator variable equal to $1$ if $[u,v]\in E_1^+\cup E_2^+$ is traversed, and $\tilde{x}_{t[u,v]}\in \{0,1\}$ be an indicator variable equal to $1$ if $[u,v]\in E_3^+$ is traversed. The TSP-based formulation $P^2_U$ is similar to $P^1_U$, and is described by the following constraints.
}
\begin{alignat}{2} 
\label{tspt0} \quad & \sum_{[s,v]\in \delta(s)}x_{t[s,v]} \geq 1, &  \qquad & \forall t\in \mathcal{T} \\
\label{tspt1} \quad & \sum_{[s,v]\in \delta(s)}(x_{t[s,v]} + \tilde{x}_{t[s,v]}) = 2, &  \qquad & \forall t\in \mathcal{T} \\
\label{tspt2} \quad & x_{t[u,v]} \geq z_{ot}, &  \qquad & \forall t\in \mathcal{T}, i\in [W_{sub}], [u,v] = e(i), o\in \{ o: V_{sub}(i) \cap L_o \neq \emptyset\} \\
\label{tspt3} \quad & \sum_{[u,v] \in \delta(u)} (x_{t[u,v]} + \tilde{x}_{t[u,v]}) = 2 y_{tu}, & \qquad & \forall t\in \mathcal{T},  u \in V_I\backslash \{ s \} \\
\label{tspt4} \quad & \sum_{[u,v] \in \delta(S)} (x_{t[u,v]}+\tilde{x}_{t[u,v]}) \geq 2 y_{tu_0}, & \qquad & \forall t\in \mathcal{T},  S \subset V_I\backslash \{ s \}, |S|\geq 2, u_0\in S \\
\label{tspt5} \quad & \sum_{t \in \mathcal{T}} z_{ot} = 1, & \qquad & \forall o \in O \\
\label{tspt6} \quad & \sum_{o \in O} b_o z_{ot} \leq B, & \qquad & \forall t \in \mathcal{T} \\
\label{tspt7} \qquad & x_{t[u,v]} \in \{ 0,1 \}, & \qquad &  \forall t \in \mathcal{T}, [u,v] \in E^+_1\cup E^+_2\\
\label{tspt8} \qquad & \tilde{x}_{t[u,v]} \in \{ 0,1 \}, & \qquad &  \forall t \in \mathcal{T}, [u,v] \in E^+_3 \\
\label{tspt9} \qquad & y_{tv} \in \{ 0,1 \}, & \qquad &  \forall t \in \mathcal{T}, v \in V_I\\
\label{tspt10} \qquad & z_{ot} \in \{ 0,1 \}, & \qquad &  \forall t \in \mathcal{T}, o \in O
\end{alignat}

\hl{A question immediately arises: is there always an optimal picking tour that can be induced by a feasible solution of $P^2_U$? In the remainder of this section, we reveal the existence of such a tour.
}

Assume that we have known the set $K_1$($K_2$) of subaisles in the first (second) block, which contains at least one product to be picked. For simplicity, we assume that $K_1\neq \emptyset$ and $K_2\neq \emptyset$. \hl{Let $i_0$ be the first subaisle in $K_1$. The following two routes, which we call S-shape routes, are considered.
}
\begin{enumerate}
    \item $r_S^1$ will first visit all subaisles in $K_1\backslash \{i_0\}$, then visit all subaisles in $K_2$ and finally visit subaisle $i_0$ (as shown in Figure~\ref{s-shape}(a)).
    \item $r_S^2$ will first visit all subaisles in $K_1$, then visit all subaisles in $K_2$ (as shown in Figure~\ref{s-shape}(b)).
\end{enumerate}

\hl{Obviously, routes of type $r_S^1$ or $r_S^2$ can always be represented by a feasible solution of $P^2_U$. Furthermore, the following theorem guarantees the existence of an optimal tour.}

\begin{figure}
\centering
\subfigure[$r_S^1$]{%
\resizebox*{5cm}{!}{\includegraphics{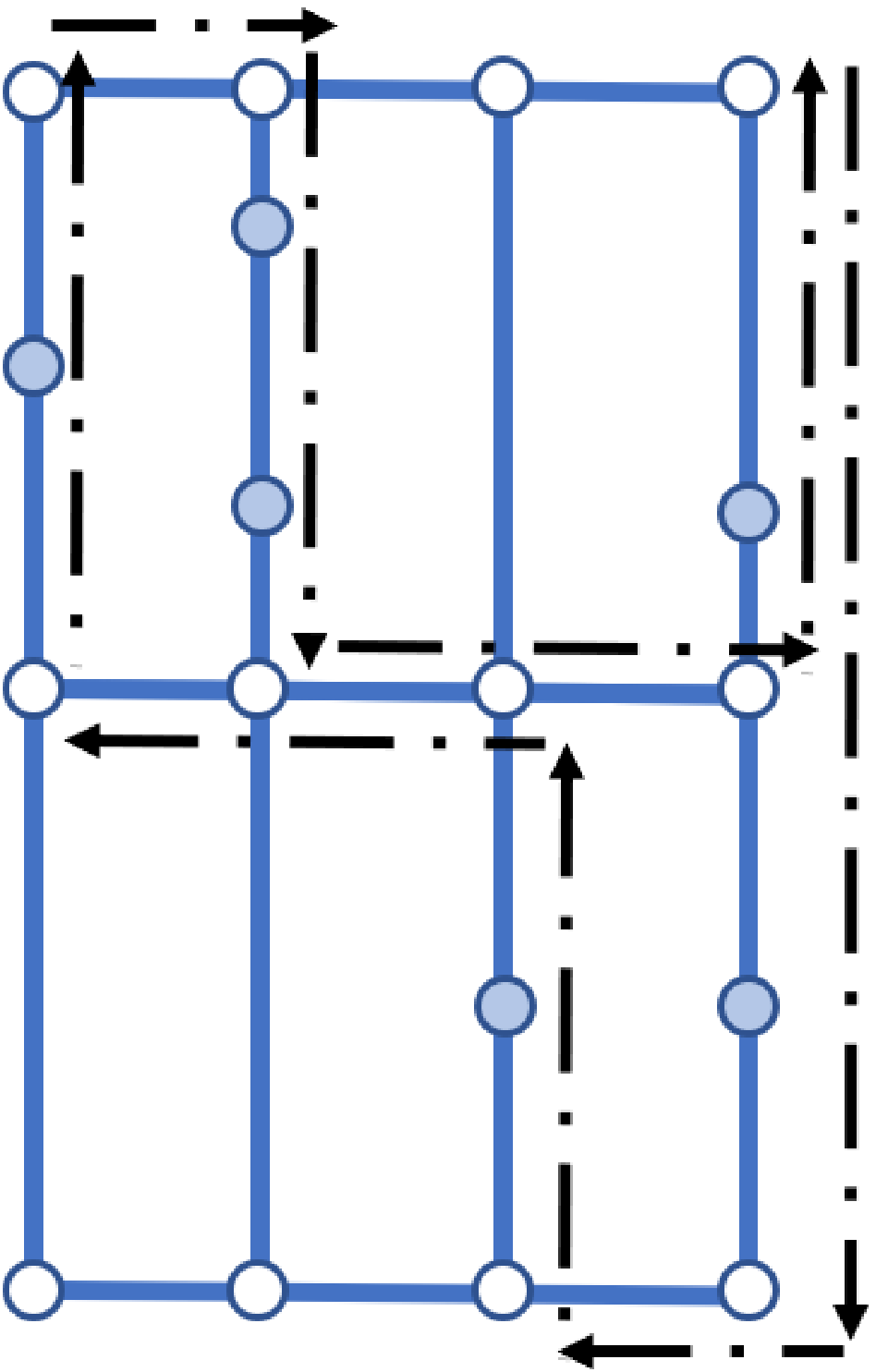}}}
\subfigure[$r_S^2$]{%
\resizebox*{5cm}{!}{\includegraphics{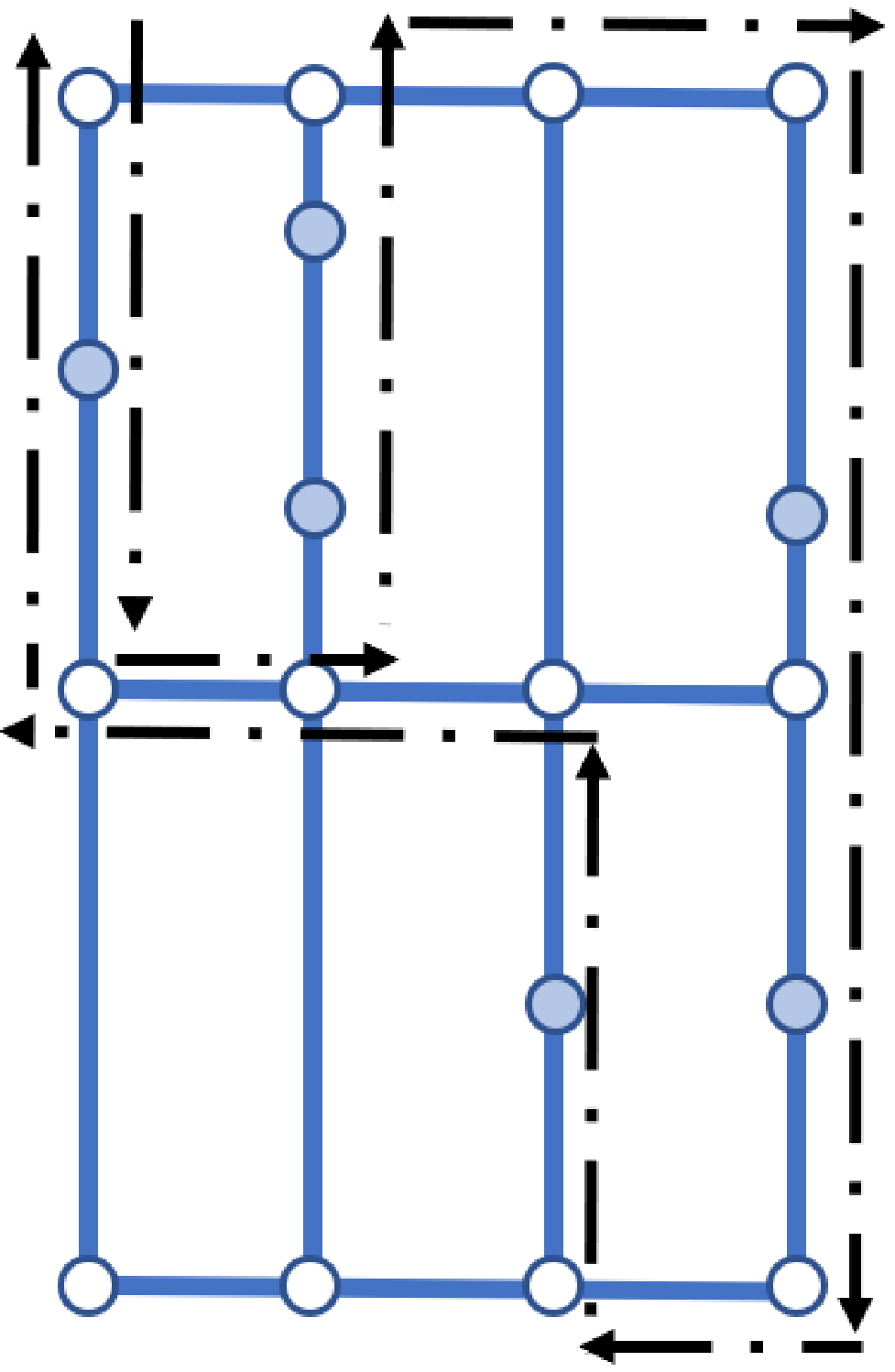}}}
\caption{The S-shape route} \label{s-shape}
\end{figure}

\begin{theorem}\label{Sshape}
\hl{There always exists an optimal picking tour which is done in an S-shape fashion.}
\end{theorem}
\begin{proof}
\hl{We first consider the case where subaisle $1\in K_1$.} Let the total vertical distance of route $r$ \hl{be $l_{r}$} and the optimal route $r^*$. Let $d$ denotes the length of a subaisle. It can be seen that $l_{r^*} \geq |K_1\cup K_2|d$. We also have the following results for \hl{$r_S^1$}:
\begin{enumerate}
  \item If $|K_1|$ and $|K_2|$ are odd, we have that $l_{r_S^1}-|K_1\cup K_2|d = 2d$ when $1\in K_1$.
  \item If $|K_1|$ and $|K_2|$ are even, we have that $l_{r_S^1}-|K_1\cup K_2|d = 0$ when $1\in K_1$.
  \item If $|K_1|$ is odd and $|K_2|$ is even, we have that $l_{r_S^1}-|K_1\cup K_2|d = d$ when $1\in K_1$.
  \item If $|K_1|$ is even and $|K_2|$ is odd, we have that $l_{r_S^1}-|K_1\cup K_2|d = d$ when $1\in K_1$.
\end{enumerate}

\hl{We now show that $r_S^1$ is an optimal picking tour.} We only deal with the case when $|K_1|$ is odd and $|K_2|$ is even. Assume by contradiction that $l_{r^*}=|K_1\cup K_2|d$, which means $r^*$ will \hl{traverse} each selected subaisle exactly once. $r^*$ can be described as a sequence of subaisles $i_1,i_2,...i_n$. If $i_1,i_k\in K_1$ and $i_2,...,i_{k-1}\in K_2$, we can find that $k$ is even. $i_1,i_2,...i_n$ can be reformulated as $Q_1=\{i_1,i_2,...,i_{k_1}\}$,$Q_2=\{i_{k_1+1},i_2,...,i_{k_2}\}$,...,$Q_{2n+1}=\{i_{k_{2n}+1},i_{k_{2n}+2},...,i_{k_{2n+1}}\}$ where $Q_{2t+1}\subset K_1$ and $Q_{2t}\subset K_2$. Note that $|Q_2|,...,|Q_{2n}|$ are even and $|K_1|$ is odd, we can assume that $|Q_1|$ is even. This implies that $r^*$ must traverse $i_{k_1}$ twice, which is a contradiction.

\hl{Similarly, we can prove that there exist an optimal picking tour which is of type $r_S^1$ or $r_S^2$ when $1\notin K_1$. This completes the proof.}   
\end{proof}

\hl{Note that} the second cross-aisle is passed through by \hl{$r_S^1$ or $r_S^2$} exactly twice, \hl{and thus,} we \hl{can further tighten the feasible region}. Let $V_S$ be the set of \hl{vertices} in the second block, i.e., $V_S=\{W'\cup \{Q_S(v):v\in W'\}\}$. We use the following constraint to ensure that the second cross-aisle can be passed through at most twice by the picker. \hl{This constraint} is very effective at reducing solution times. 

\begin{alignat}{2} 
\label{less2con} \quad & \sum_{[u,v]\in \delta(V_S)}(x_{t[u,v]}+\tilde{x}_{t[u,v]}) \leq 2, &  \qquad & \forall t\in \mathcal{T} 
\end{alignat}

\hl{We finish this section with the following corollary.}

\begin{corollary}
There always exists an optimal solution satisfying constraints~(\ref{less2con}).
\end{corollary}
\begin{proof}
This follows immediately from Theorem~\ref{Sshape}.
\end{proof}

\section{Computational results}\label{computrel}

The experiments were performed on an AMD Ryzen 7 4800H @2.90 GHz processor and 32 GB of RAM. The code was written in \hl{Python} and GUROBI 9.1.1 was used as the mixed-integer solver. The performances of our MIP formulations and additional constraints are tested by comparing the computational difficulties to find optimal solutions.

\subsection{Test problems}

Our formulations are tested over the publicly available benchmark instances at \url{http://www.dcc.ufmg.br/∼arbex/orderpicking.html}. It comes from a database of anonymized customer purchases over two years for a chain of supermarkets. A single order is generated by combing the purchases of a customer over the first $\Delta$ days with $\Delta\in \{5,10,20\}$. The warehouse layouts considered in our experiments are similar to that of \citep{VALLE2017817}; the slight difference is that we assume the origin is right from the first artificial location while they set the distance from the origin to the first artificial vertex is 4 meters. Our formulations for their warehouse layout are essentially the same. \hl{We set the available capacity of the picking vehicle $B=8$.} For every test instance, we define $T$ by solving a bin-packing problem.

\subsection{Computational results}

In this section, we compare our formulations \hl{with that} of \citep{VALLE2017817}. The formulation for the JOBPRP (resp., for the no-reversal JOBPRP) presented by \citep{VALLE2017817} is denoted as $P_O$ (resp., \hl{$P_{U}$}). \hl{To further improve our formulation, we take into account existing} constraints including \hl{aisles cuts,} artificial vertex reversal constraints \citep{VALLE2017817} and column \hl{inequalities} \citep{10.1007/s10107-006-0081-5}. \hl{To have a better comparison, we also use column inequalities to improved $P_O$. We do not present experimental results to verify the effectiveness of these constraints that existing studies have already illustrated. Details regarding test formulations are shown in Table~\ref{manymodels}.}

\begin{longtable}{p{2cm}p{11cm}}
\caption{Details about test formulations}
\label{manymodels}\\
\toprule

Notation                 & Explanation\\                                    
\midrule

$P_O$                     & $P_{basic}$ with the valid inequalities defined in \citep{VALLE2017817} and column inequalities \\
$P_G^+$  & $P_G$ with aisles
cuts, artificial vertex reversal constraints and column inequalities                                 \\
$P_F^+$  &  $P_F$ with artificial vertex reversal constraints and column inequalities                                      \\
$P_U$   &  the no-reversal formulation in \citep{VALLE2017817} with column inequalities                                                  \\
$P^{1+}_U$ & $P^{1}_U$ with column inequalities                       \\
$P^{2+}_U$ & $P^{2}_U$ with column inequalities \\            
\bottomrule
\end{longtable}
\hl{Remark that all formulations except $P_F^+$ are non-compact. The exponentially many constraints are generated sequentially in a branch-and-cut framework. For each candidate integral solution, we add constraints when the graph for some picker is disconnected. The connectivity condition is verified by a depth-first search. In addition, $P_F$ seems unable to benefit from aisles cuts and therefore $P_F^+$ does not include aisles cuts.}

In Table~\ref{exp1}, we compare formulation $P_O$, \hl{$P_G^+$ and $P_F^+$} on the selected instances by setting a time limit of 2400 s. Column $O$ corresponds to the number of orders. Column $T$ corresponds to the total time. Columns $UB$ and $LB$ represent the best upper and lower bounds obtained at the end of the search, respectively, when either the instance was solved to prove optimality or the time limit has hit. GAP is defined as $100\% \times \frac{UB-LB}{UB}$.

\begin{table}
\tbl{Comparison of the branch-and-cut algorithm based on formulations $P_O$,$P_G^+$ and $P_F^+$.}
{\begin{tabular}{llcccccccccccccc} \toprule
 & &\multicolumn{2}{l}{$P_O$} 
 & & & & \multicolumn{2}{l}{$P_G^+$}
 & & & & \multicolumn{2}{l}{$P_F^+$}\\ \cmidrule{3-6} \cmidrule{8-11} \cmidrule{13-16}
 $\Delta$ & O & T(seconds) & UB & LB & GAP(\%) & & T(seconds) & UB & LB & GAP(\%) & & T(seconds) & UB & LB & GAP(\%) \\ \midrule
 5 & 5 & 0.2 & 346 & 346 & 0 & & 0.24 & 346 & 346 & 0 & & 0.51 & 346 & 346 & 0\\
   & 10 & 2.91 & 578 & 578 & 0 & & 1.5 & 578 & 578 & 0 & & 92 & 578 & 578 & 0\\
   & 15 & 15 & 650 & 650 & 0 & & 7.8 & 650 & 650 & 0 & & 41 & 650 & 650 & 0\\
   & 16 & 390 & 766 & 766 & 0 & & 37 & 766 & 766 & 0 & & 236 & 766 & 766 & 0\\
   & 17 & 90 & 802 & 802 & 0 & & 30 & 802 & 802 & 0 & & 230 & 802 & 802 & 0 \\
   & 18 & 1821 & 840 & 840 & 0 & & 81 & 840 & 840 & 0 & & 691 & 840 & 840 & 0 \\
   & 19 & 2400 & 856 & 851 & 0.6 & & 135 & 856 & 856 & 0 & & 690 & 856 & 856 & 0 \\
   & 20 & 2400 & 906 & 758 & 16.3 & & 86 & 864 & 864 & 0 & & 770 & 864 & 864 & 0\\
   & 21 & 2400 & 892 & 884 & 0.9 & & 136 & 892 & 892 & 0 & & 1585 & 892 & 892 & 0 \\
   & 22 & 2400 & 902 & 868 & 3.8 & & 171 & 892 & 892 & 0 & & 1386 & 892 & 892 & 0 \\
   & 23 & 2400 & 912 & 877 & 3.8 & & 290 & 908 & 908 & 0 & & 2400 & 908 & 901 & 0.8 \\
   & 24 & 2400 & 1118 & 723 & 35.3 & & 2400 & 1059 & 925 & 12.7 & & 2400 & 1056 & 862 & 18.3 \\
   & 25 & 2400 & 1104 & 815 & 26.2 & & 2400 & 1102 & 954 & 13.4 & & 2400 & 1104 & 869 & 21.3 \\
   & 30 & 2400 & 1200 & 843 & 29.8 & & 2400 & 1206 & 961 & 20.3 & & 2400 & 1206 & 864 & 28.4 \\
\midrule
10 & 5 & 0.06 & 368 & 368 & 0 & & 0.06 & 368 & 368 & 0 & & 0.12 & 368 & 368 & 0\\
    & 10 & 40 & 656 & 656 & 0 & & 6.31 & 656 & 656 & 0 & & 34 & 656 & 656 & 0\\
    & 15 & 195 & 874 & 874 & 0 & & 59 & 874 & 874 & 0 & & 263 & 874 & 874 & 0\\
    & 16 & 178 & 926 & 926 & 0 & & 65 & 926 & 926 & 0 & & 311 & 926 & 926 & 0\\
    & 17 & 1188 & 960 & 960 & 0 & & 123 & 960 & 960 & 0 & & 996 & 960 & 960 & 0\\
    & 18 & 892 & 970 & 970 & 0 & & 106 & 970 & 970 & 0 & & 1112 & 970 & 970 & 0\\
    & 19 & 375 & 978 & 978 & 0 & & 166 & 978 & 978 & 0 & & 607 & 978 & 978 & 0\\
    & 20 & 454 & 984 & 984 & 0 & & 209 & 984 & 984 & 0 & & 1145 & 984 & 984 & 0\\
    & 21 & 320 & 990 & 990 & 0 & & 143 & 990 & 990 & 0 & & 1462 & 990 & 990 & 0\\
    & 22 & 1400 & 1000 & 1000 & 0 & & 180 & 1000 & 1000 & 0 & & 1880 & 1000 & 1000 & 0\\
    & 23 & 2400 & 1162 & 959 & 17.5 & & 2400 & 1140 & 978 & 14.2 & & 2400 & 1132 & 881 & 22.2\\
    & 24 & 2400 & 1218 & 824 & 32.3 & & 2400 & 1162 & 1012 & 12.9 & & 2400 & 1188 & 907 & 23.7\\
    & 25 & 2400 & 1192 & 935 & 21.6 & & 2400 & 1220 & 1006 & 17.5 & & 2400 & 1182 & 966 & 18.3\\
    & 30 & 2400 & 1326 & 955 & 28 & & 2400 & 1320 & 1083 & 18 & & 2400 & 1284 & 969 & 24.5\\
\midrule
20 & 5 & 19 & 570 & 570 & 0 & & 1.03 & 570 & 570 & 0 & & 7.5 & 570 & 570 & 0\\
    & 10 & 68 & 912 & 912 & 0 & & 25 & 912 & 912 & 0 & & 119 & 912 & 912 & 0\\
    & 15 & 2400 & 1026 & 1009 & 1.7 & & 58 & 1022 & 1022 & 0 & & 255 & 1022 & 1022 & 0\\
    & 16 & 2400 & 1206 & 1085 & 10 & & 745 & 1200 & 1200 & 0 & & 1409 & 1200 & 1200 & 0\\
    & 17 & 2400 & 1292 & 1029 & 20.4 & & 2059 & 1250 & 1250 & 0 & & 2400 & 1250 & 1161 & 7.1\\
    & 18 & 2400 & 1324 & 1070 & 19.2 & & 1166 & 1288 & 1288 & 0 & & 2400 & 1310 & 1145 & 12.6\\
    & 19 & 2400 & 1356 & 1093 & 19.4 & & 2400 & 1326 & 1146 & 13.6 & & 2400 & 1346 & 1108 & 17.7\\
    & 20 & 2400 & 1334 & 1087 & 18.5 & & 2400 & 1340 & 1292 & 3.6 & & 2400 & 1340 & 1189 & 11.3\\
    & 21 & 2400 & 1554 & 960 & 38.2 & & 2400 & 1542 & 1145 & 25.7 & & 2400 & 1534 & 1078 & 29.7\\
    & 22 & 2400 & 1702 & 909 & 46.6 & & 2400 & 1578 & 1162 & 26.4 & & 2400 & 1558 & 1135 & 27.2\\
    & 23 & 2400 & 1638 & 1030 & 37.1 & & 2400 & 1624 & 1130 & 30.4 & & 2400 & 1620 & 1115 & 31.2\\
    & 24 & 2400 & 1672 & 968 & 42.1 & & 2400 & 1640 & 1193 & 27.3 & & 2400 & 1652 & 1040 & 37\\
    & 25 & 2400 & -\textsuperscript{*} & 930 & - & & 2400 & 1648 & 1141 & 30.8 & & 2400 & 1644 & 1083 & 34.1\\
    & 30 & 2400 & - & 987 & - & & 2400 & 1900 & 1125 & 40.8 & & 2400 & 1944 & 988 & 49.2\\
\bottomrule
\end{tabular}}
\tabnote{\textsuperscript{*}The symbol '-' shows that GUROBI failed to find a feasible solution.}  
\label{exp1}
\end{table}

\hl{The} branch-and-\hl{cut} algorithm based on \hl{$P_G^+$} managed to solve most instances to proven optimality. Furthermore, it obtained the lowest gap or had the shortest computing time. Thus, we can state that \hl{$P_G^+$} outperforms the existing formulation $P_O$. However, \hl{$P_F^+$} is not as strong as \hl{$P_G^+$}. Although it provided a better gap (than $P_O$) for some instances (for example, instance with $\Delta=5,O\in[19,30]$), it performs poorly when $\Delta=10$. Furthermore, an increasing number of orders \hl{cause} a fast \hl{increase} in the solution time, even for \hl{$P_G^+$}. One possible reason is that as the number of order pickers increases, the number of symmetry branches in the search tree grows exponentially. \hl{In fact, even the relaxation $min\{f(x):(x,z,\alpha,\beta) \in P_{sub}, constraints~(\ref{bs1})-(\ref{bs2}),(\ref{bs5})-(\ref{bs8}),(\ref{bs10})\}$ (see Figure~\ref{feasol}(c)) is very difficult to solve when there are a large number of orders (and it cannot figure out the route for each picker).}

\hl{We also compare our formulations with two commonly used heuristics for order batching problem: the seed algorithm and the Clarke and Wright algorithm(II) \citep{doi:10.1080/002075499191094}. For each batch, we find an S-shape route to estimate the traveling distance (see section~\ref{tsp2block}). Typically, the Seed and CWII can provide feasible solutions within several seconds. However, the solutions seem to be far from optimal. Numerical results are given in Table~\ref{heutst} where the last column shows the currently best known solution.}

\begin{table}
\tbl{Experimental results for heuristic solution approaches}
{\begin{tabular}{lccccccccccccc} \toprule
  & &\multicolumn{2}{l}{$\Delta = 5$} 
 & & & \multicolumn{2}{l}{$\Delta = 10$}
 & & & \multicolumn{2}{l}{$\Delta = 20$} &\\ \cmidrule{2-4} \cmidrule{6-8} \cmidrule{10-12}
 O & Seed & CWII & Best & & Seed & CWII & Best & & Seed & CWII & Best \\ \midrule
 5 & 382 & 382 & 346 & & 382 & 382 & 368 & & 620 & 620 & 570  \\
 10 & 636 & 636 & 578 & & 726 & 726 & 656 & & 982 & 982 & 912  \\
 15 & 724 & 764 & 650 & & 922 & 922 & 874 & & 1146 & 1146 & 1022  \\
 16 & 910 & 882 & 766 & & 1058 & 1058 & 926 & & 1372 & 1372 & 1200  \\
 17 & 902 & 930 & 802 & & 1058 & 1058 & 960 & & 1488 & 1450 & 1250  \\
 18 & 942 & 980 & 840 & & 1058 & 1058 & 970 & & 1490 & 1490 & 1288  \\
 19 & 980 & 1030 & 856 & & 1096 & 1096 & 978 & & 1490 & 1490 & 1326  \\
 20 & 980 & 1018 & 864 & & 1096 & 1096 & 984 & & 1488 & 1488 & 1334  \\
 21 & 1018 & 1018 & 892 & & 1146 & 1146 & 990 & & 1716 & 1728 & 1534  \\
 22 & 1018 & 1018 & 892 & & 1146 & 1146 & 1000 & & 1754 & 1766 & 1558  \\
 23 & 1058 & 1058 & 908 & & 1284 & 1284 & 1132 & & 1812 & 1832 & 1620  \\
 24 & 1186 & 1158 & 1056 & & 1332 & 1332 & 1162 & & 1872 & 1872 & 1640  \\
 25 & 1246 & 1312 & 1102 & & 1372 & 1362 & 1182 & & 1872 & 1872 & 1644  \\
 30 & 1440 & 1360 & 1200 & & 1452 & 1490 & 1284 & & 2128 & 2166 & 1900  \\
\bottomrule
\end{tabular}}
\label{heutst}
\end{table}

\hl{We analyze the efficiency of the basic cuts and the single traversing constraints by adding them to $P_O$, $P_G^+$ and $P_F^+$. Note that the basic cuts are generated by performing a depth-first search algorithm, and the total running time is typically less than 0.5 seconds (0.03s-0.5s). Therefore we do not need to take into account the processing time of constructing the basic cuts. The original model and the strengthened model are compared by counting winning instances. An instance is a winner for model A compared with model B, if
}
\begin{enumerate}
  \item A finished within the time limit and B did not finish or required a \hl{longer} CPU time \hl{or}
  \item A \hl{obtained} a lower gap than B.
\end{enumerate} 
If the difference between the times or gaps \hl{is} below 1 s \hl{or 0.1\%}, respectively, the instance is not counted. \hl{For example, we can compare $P_O$ and $P^+_F$ in Table~\ref{exp1} by only considering the instances with $\Delta=5$. Then we can observe that $P_O$ has $3$ winners and $P_F^+$ has $10$ winners. Table~\ref{exp2},\ref{exp3} show the impacts of adding these additional constraints. Except the aggregated results, we also provide more detailed results in the appendix.}
\begin{table}
\tbl{Behavior of the basic cuts}
{\begin{tabular}{lccccc} \toprule
 & & $P_O$ & & & $P_O +$ basic cuts  \\ \midrule
$\Delta$ & & Win Rate(\%) & & & Win Rate(\%)  \\ \midrule
5 & & 69.2 & & & 30.8  \\
10 & & 46.2 & & & 53.8  \\
20 & & 30.8 & & & 69.2  \\
\midrule
 & & $P_G^+$ & & & $P_G^+ +$ basic cuts  \\ \midrule
$\Delta$ & & Win Rate(\%) & & & Win Rate(\%)  \\ \midrule
5 & & 25 & & & 75  \\
10 & & 41.7 & & & 58.3  \\
20 & & 46.2 & & & 53.8  \\
\midrule
 & & $P_F^+$ & & & $P_F^+ +$ basic cuts  \\ \midrule
$\Delta$ & & Win Rate(\%) & & & Win Rate(\%)  \\ \midrule
5 & & 58.3 & & & 41.7  \\
10 & & 53.8 & & & 46.2  \\
20 & & 71.4 & & & 28.6  \\
\bottomrule
\end{tabular}}
\label{exp2}
\end{table}

\begin{table}
\tbl{Behavior of the single traversing constraints}
{\begin{tabular}{lccccc} \toprule
 & & $P_O$ & & & $P_O +$ single traversing constraints \\ \midrule
$\Delta$ & & Win Rate(\%) & & & Win Rate(\%)  \\ \midrule
5 & & 61.5 & & & 38.5  \\
10 & & 53.8 & & & 46.2  \\
20 & & 58.3 & & & 41.7  \\
\midrule
 & & $P_G^+$ & & & $P_G^+ +$ single traversing constraints  \\ \midrule
$\Delta$ & & Win Rate(\%) & & & Win Rate(\%)  \\ \midrule
5 & & 50 & & & 50  \\
10 & & 25 & & & 75  \\
20 & & 30.8 & & & 69.2  \\
\midrule
 & & $P_F^+$ & & & $P_F^+ +$ single traversing constraints  \\ \midrule
$\Delta$ & & Win Rate(\%) & & & Win Rate(\%)  \\ \midrule
5 & & 23.1 & & & 76.9  \\
10 & & 41.7 & & & 58.3  \\
20 & & 58.3 & & & 41.7  \\
\bottomrule
\end{tabular}}
\label{exp3}
\end{table}

\hl{Table~\ref{exp2} shows the efficiency of the basic cuts. For formulation $P_O$, the basic cuts can improve at most $69.2\%$ instances (when $\Delta=20$). Similarly, the basic cuts can improve at most $75\%$ instances for $P_G^+$ (when $\Delta=5$). However, due to being compact and not requiring an explicit branch-and-cut implementation, $P_F^+$ seems to benefit less from the basic cuts. Table~\ref{exp3} shows the efficiency of the single traversing constraints. Although both $P_G^+$ (at most $75\%$ instances) and $P_F^+$ (at most $76.9$ instances) are able to benefit from the single traversing constraints, more than half of instances $P_O$ cannot be improved by these constraints. We also note that the basic cuts and the single traversing constraints} sometimes increase the solution time, which could be due to the interaction of these constraints and some built-in general-purpose cuts. Furthermore, we \hl{believe} that the single traversing \hl{constraints} should be given more consideration; This can lead to a much smaller feasible region and may induce other constraints or formulations.

In Table~\ref{exp4}, we compare formulation \hl{$P_U$} and \hl{$P_U^{1+}$} for a single-block warehouse setting a time limit of 300 s. We reduce the time limit, mainly because the no-reversal JOBPRP is much simpler than the JOBPRP. The branch-and-\hl{cut} algorithm based on \hl{$P_U^{1+}$} can solve all instances within several seconds. The main reason \hl{could be that} we successfully cut many symmetric solutions by designing an auxiliary graph. Similarly, we compare formulation \hl{$P_U$} and \hl{$P_U^{2+}$} for a 2-block warehouse in Table~\ref{exp5}, and \hl{$P_U^{2+}$} still outperforms \hl{$P_U$} for all instances.

\begin{table}
\tbl{Comparison of the branch-and-cut algorithm based on formulations $P_U$ and $P_U^{1+}$ for a single-block warehouse.}
{\begin{tabular}{llccccccccc} \toprule
 & &\multicolumn{2}{l}{$P_U$} 
 & & & & \multicolumn{2}{l}{$P_U^{1+}$}
 \\ \cmidrule{3-6} \cmidrule{8-11} 
 $\Delta$ & O & T(seconds) & UB & LB & GAP(\%) & & T(seconds) & UB & LB & GAP(\%) \\ \midrule
5 & 5 & 0.02 & 358 & 358 & 0 & & 0.01 & 358 & 358 & 0 \\
  & 10 & 0.13 & 634 & 634 & 0 & & 0.05 & 634 & 634 & 0 \\
  & 15 & 0.05 & 716 & 716 & 0 & & 0.01 & 716 & 716 & 0 \\
  & 20 & 1.72 & 982 & 982 & 0 & & 0.16 & 982 & 982 & 0 \\
  & 21 & 26 & 1064 & 1064 & 0 & & 0.21 & 1064 & 1064 & 0 \\
  & 22 & 41 & 1064 & 1064 & 0 & & 0.25 & 1064 & 1064 & 0 \\
  & 23 & 20 & 1064 & 1064 & 0 & & 0.27 & 1064 & 1064 & 0 \\
  & 24 & 300 & 1248 & 1140 & 8.7 & & 1.64 & 1248 & 1248 & 0 \\
  & 25 & 98 & 1258 & 1258 & 0 & & 2.21 & 1258 & 1258 & 0 \\
  & 26 & 64 & 1268 & 1268 & 0 & & 1.3 & 1268 & 1268 & 0 \\
  & 27 & 215 & 1278 & 1278 & 0 & & 1.17 & 1278 & 1278 & 0 \\
  & 28 & 235 & 1330 & 1330 & 0 & & 2.62 & 1330 & 1330 & 0 \\
  & 29 & 300 & 1350 & 1340 & 0.7 & & 1.2 & 1350 & 1350 & 0 \\
  & 30 & 300 & 1350 & 1304 & 0.3 & & 1.66 & 1350 & 1350 & 0 \\
\midrule
10 & 5 & 0.01 & 358 & 358 & 0 & & 0.01 & 358 & 358 & 0 \\
  & 10 & 0.07 & 716 & 716 & 0 & & 0.01 & 716 & 716 & 0 \\
  & 15 & 1.26 & 972 & 972 & 0 & & 0.19 & 972 & 972 & 0 \\
  & 20 & 1.27 & 992 & 992 & 0 & & 0.23 & 992 & 992 & 0 \\
  & 21 & 10 & 1064 & 1064 & 0 & & 0.18 & 1064 & 1064 & 0 \\
  & 22 & 2.56 & 1064 & 1064 & 0 & & 0.25 & 1064 & 1064 & 0 \\
  & 23 & 73 & 1248 & 1248 & 0 & & 1.22 & 1248 & 1248 & 0 \\
  & 24 & 70 & 1248 & 1248 & 0 & & 0.72 & 1248 & 1248 & 0 \\
  & 25 & 4.53 & 1268 & 1268 & 0 & & 0.65 & 1268 & 1268 & 0 \\
  & 26 & 35 & 1268 & 1268 & 0 & & 0.66 & 1268 & 1268 & 0 \\
  & 27 & 214 & 1330 & 1330 & 0 & & 0.63 & 1330 & 1330 & 0 \\
  & 28 & 119 & 1340 & 1340 & 0 & & 0.52 & 1340 & 1340 & 0 \\
  & 29 & 77 & 1340 & 1340 & 0 & & 0.7 & 1340 & 1340 & 0 \\
  & 30 & 82 & 1340 & 1340 & 0 & & 0.6 & 1340 & 1340 & 0 \\
\midrule
20 & 5 & 0.09 & 706 & 706 & 0 & & 0.03 & 706 & 706 & 0 \\
    & 10 & 0.88 & 992 & 992 & 0 & & 0.1 & 992 & 992 & 0 \\
    & 15 & 0.88 & 1074 & 1074 & 0 & & 0.12 & 1074 & 1074 & 0 \\
    & 20 & 31 & 1422 & 1422 & 0 & & 0.34 & 1422 & 1422 & 0 \\
    & 21 & 300 & 1698 & 1675 & 1.4 & & 1.92 & 1698 & 1698 & 0 \\
    & 22 & 300 & 1698 & 1672 & 1.5 & & 1.04 & 1698 & 1698 & 0 \\
    & 23 & 300 & 1770 & 1699 & 4 & & 1.21 & 1770 & 1770 & 0 \\
    & 24 & 300 & 1770 & 1693 & 4.4 & & 1.13 & 1770 & 1770 & 0 \\
    & 25 & 300 & 1780 & 1736 & 2.5 & & 0.88 & 1780 & 1780 & 0 \\
    & 26 & 300 & 1780 & 1774 & 0.3 & & 0.66 & 1780 & 1780 & 0 \\
    & 27 & 33 & 1780 & 1780 & 0 & & 1.85 & 1780 & 1780 & 0 \\
    & 28 & 300 & 2056 & 1903 & 7.4 & & 4.66 & 2056 & 2056 & 0 \\
    & 29 & 300 & 2056 & 1652 & 19.6 & & 5.19 & 2056 & 2056 & 0 \\
    & 30 & 300 & 2056 & 1780 & 13.4 & & 4.47 & 2056 & 2056 & 0 \\
\bottomrule
\end{tabular}}
\label{exp4}
\end{table}

\begin{table}
\tbl{Comparison of the branch-and-cut algorithm based on formulations $P_U$ and $P_U^{2+}$ for a 2-block warehouse.}
{\begin{tabular}{llccccccccc} \toprule
 & &\multicolumn{2}{l}{$P_U$} 
 & & & & \multicolumn{2}{l}{$P_U^{2+}$}
 \\ \cmidrule{3-6} \cmidrule{8-11} 
 $\Delta$ & O & T(seconds) & UB & LB & GAP(\%) & & T(seconds) & UB & LB & GAP(\%) \\ \midrule
5 & 5 & 0.05 & 382 & 382 & 0 & & 0.02 & 382 & 382 & 0 \\
  & 10 & 0.51 & 608 & 608 & 0 & & 0.2 & 608 & 608 & 0 \\
  & 15 & 1.47 & 696 & 696 & 0 & & 0.22 & 696 & 696 & 0 \\
  & 20 & 20 & 940 & 940 & 0 & & 1.91 & 940 & 940 & 0 \\
  & 21 & 5.34 & 940 & 940 & 0 & & 1.73 & 940 & 940 & 0 \\
  & 22 & 7.18 & 940 & 940 & 0 & & 1.2 & 940 & 940 & 0 \\
  & 23 & 8.2 & 950 & 950 & 0 & & 1.54 & 950 & 950 & 0 \\
  & 24 & 273 & 1108 & 1108 & 0 & & 31 & 1108 & 1108 & 0 \\
  & 25 & 152 & 1146 & 1146 & 0 & & 41 & 1146 & 1146 & 0 \\
  & 26 & 300 & 1194 & 1118 & 6.4 & & 45 & 1176 & 1176 & 0 \\
  & 27 & 300 & 1206 & 1185 & 1.7 & & 52 & 1206 & 1206 & 0 \\
  & 28 & 300 & 1234 & 1151 & 6.7 & & 67 & 1206 & 1206 & 0 \\
  & 29 & 300 & 1254 & 1184 & 5.6 & & 67 & 1254 & 1254 & 0 \\
  & 30 & 300 & 1254 & 1151 & 8.2 & & 48 & 1254 & 1254 & 0 \\
\midrule
10 & 5 & 0.02 & 382 & 382 & 0 & & 0.02 & 382 & 382 & 0 \\
  & 10 & 1.29 & 724 & 724 & 0 & & 0.29 & 724 & 724 & 0 \\
  & 15 & 4.74 & 922 & 922 & 0 & & 3.08 & 922 & 922 & 0 \\
  & 20 & 8 & 1020 & 1020 & 0 & & 0.97 & 1020 & 1020 & 0 \\
  & 21 & 132 & 1058 & 1058 & 0 & & 1.42 & 1058 & 1058 & 0 \\
  & 22 & 26 & 1058 & 1058 & 0 & & 1.13 & 1058 & 1058 & 0 \\
  & 23 & 237 & 1214 & 1214 & 0 & & 22 & 1214 & 1214 & 0 \\
  & 24 & 266 & 1254 & 1254 & 0 & & 19 & 1254 & 1254 & 0 \\
  & 25 & 213 & 1254 & 1254 & 0 & & 14 & 1254 & 1254 & 0 \\
  & 26 & 300 & 1302 & 1283 & 1.5 & & 17 & 1302 & 1302 & 0 \\
  & 27 & 300 & 1342 & 1302 & 3 & & 35 & 1342 & 1342 & 0 \\
  & 28 & 300 & 1352 & 1231 & 8.9 & & 23 & 1352 & 1352 & 0 \\
  & 29 & 300 & 1352 & 1264 & 6.5 & & 20 & 1352 & 1352 & 0 \\
  & 30 & 300 & 1352 & 1255 & 7.2 & & 22 & 1352 & 1352 & 0 \\
\midrule
20 & 5 & 0.2 & 620 & 620 & 0 & & 0.09 & 620 & 620 & 0 \\
   & 10 & 2.68 & 982 & 982 & 0 & & 0.69 & 982 & 982 & 0 \\
   & 15 & 59 & 1108 & 1108 & 0 & & 0.67 & 1108 & 1108 & 0 \\
   & 20 & 300 & 1432 & 1411 & 1.5 & & 11 & 1430 & 1430 & 0 \\
   & 21 & 300 & 1638 & 1513 & 7.6 & & 154 & 1626 & 1626 & 0 \\
   & 22 & 300 & 1646 & 1526 & 7.3 & & 114 & 1626 & 1626 & 0 \\
   & 23 & 300 & 1684 & 1534 & 8.9 & & 162 & 1648 & 1648 & 0 \\
   & 24 & 300 & 1744 & 1600 & 8.2 & & 234 & 1708 & 1708 & 0 \\
   & 25 & 300 & 1718 & 1565 & 8.9 & & 125 & 1714 & 1714 & 0 \\
   & 26 & 300 & 1746 & 1562 & 10.5 & & 148 & 1716 & 1716 & 0 \\
   & 27 & 300 & 1756 & 1589 & 9.5 & & 179 & 1736 & 1736 & 0 \\
   & 28 & 300 & 1932 & 1430 & 26 & & 300 & 1894 & 1538 & 18.8 \\
   & 29 & 300 & 1990 & 1444 & 27.4 & & 300 & 1932 & 1629 & 15.7 \\
   & 30 & 300 & 2066 & 1476 & 28.6 & & 300 & 1990 & 1734 & 12.9 \\
\bottomrule
\end{tabular}}
\label{exp5}
\end{table}

\section{Conclusions}\label{conclusion}

\hl{In this article, we investigate the JOBPRP, which is pivotal for the efficiency of order picking operations. To fully utilize the structure of the warehouse, we reconstruct the connectivity constraints. The obtained formulations, which consider separately the graph properties of picking locations and artificial locations, can significantly improve computational performance. We also provide two types of relevant additional constraints: one aims at dealing with batching decisions and routing decisions in an integrated way; the other aims at cutting off a subset of feasible solutions by the property of an optimal routing. Additionally, we consider the optimal routing for the no-reversal special case of this problem and propose TSP-based formulations. Our experimental results also show that the TSP-based formulations are very powerful and can significantly improve solution quality.} 

\hl{There are several potential topics for future research. First, graph-based mathematical formulations should consider the warehouse structure, which implies a need for polyhedral studies of different warehouses. For example, one might investigate the graph representation and the associated polytope for the HappyChic warehouse considered by \citep{BRIANT2020497}, which is slightly different from the rectangular warehouse considered in this paper. Second, one might improve traditional heuristic algorithms by analyzing the property of optimal solutions. Third, both the routing and batching problems suffer severely from the presence of symmetry. If we treat the batching problem as a partitioning problem, we can find many symmetry breaking methods (for example, column inequalities \citep{10.1007/s10107-006-0081-5}). One might make use of these symmetry breaking methods to improve different heuristics or exact methods. Fourth, as no-reversal routes are easy to implement in practice, it might be worthwhile to pay more attention to this special case. \citep{ARBEXVALLE2020460} demonstrated the feasibility of using easy-to-solve approximation programs to obtain high-quality no-reversal solutions. One might build up an approximation model that only considers some features of a feasible solution, and might study the accuracy of the estimation. 
}

\section*{Data availability statement} 
The data that support the findings of this study are available from the corresponding author, C.H. Gao, upon reasonable request.

\bibliographystyle{tfcad}
\bibliography{interactcadsample}

\newpage

\section*{Appendix}

\begin{longtable}{p{5cm}p{10cm}}
Notation                                                & Explanation                                                                                                                         \\
\midrule

\multicolumn{2}{l}{Sets}                                                                                                                                                                      \\
$\mathcal{T}$                                                       & set of available trolleys                                                                                                           \\
$O$                                                       & set of orders                                                                                                                       \\
$L_o$                                                   & set of picking locations of order $o$                                                                                                 \\
$V$                                                       & set of all locations                                                                                                                \\
$V_L$                                                    & set of picking locations                                                                                                            \\
$V_I$                                                    & set of artificial locations                                                                                                         \\
$V_{sub}(i)$                                               & set of picking locations within subaisle $i$                                                                                          \\
$\tilde{E}$                                                 & set of directed edges connecting neighboring locations                                                                              \\
$\tilde{E}'$                                                & set of directed edges connecting neighboring artificial locations while ignoring picking locations                                 \\
$\delta(S)$ & set of undirected edges with one end in set $S$                                                                             \\
$\delta^+(S) / \delta^-(S)$ & set of directed edges in $\tilde{E}$ that leave/enter set $S$                                                                             \\
$\eta^+(S) / \eta^-(S)$     & set of directed edges in $\tilde{E}'$ that leave/enter set $S$                                                                            \\
{}$W_{sub}${}                                            & the number of subaisles                                                                                                                    \\
\midrule

\multicolumn{2}{l}{Constants}                                                                                                                                                                 \\
$s$                                                       & the origin of the warehouse                                                                                                         \\
$(u,v)$                                                       &  the ordered pair of location u and
location v, which represents a directed edge                                           
\\
$[u,v]$                                                       &  the unordered pair of location u and
location v, which represents an undirected edge        \\
$f(i)/l(i)$                                               & the northern/southern artificial location of subaisle $i$                                                                             \\
$n(v)/s(v)$                                               & the adjacent northern/southern location of $v$                                                                                        \\
$Q_N(v)/Q_S(v)/Q_E(v)/Q_W(v)$                         & the adjacent northern/southern/eastern/western artificial location of artificial location $v$                                         \\
$b_o$                                                   & size of order $o$                                                                                                                     \\
$B$                                                       & available capacity of a trolley                                                                                                     \\
\midrule

\multicolumn{2}{l}{Variables}                                                                                                                                                                 \\
$x_{tuv}$                                                  & Binary variable that takes value 1 if and only if $(u,v)$ ($\in\tilde{E}$) is traversed by walk $t$                                                \\
$y_{tv}$                                                   & Binary variable that takes value 1 if and only if trolley $t$ visits location $v$                                                                 \\
$z_{ot}$                                                   & Binary variable that takes value 1 if and only if trolley $t$ picks order $o$                                                                     \\
$\alpha_{tv}/\beta_{tv}$                                   & Binary variable that takes value 1 only if there exists a straight path connecting the northern/southern artificial location and $v$ in walk $t$                                             \\
$\gamma_{tuv}$                                             & Binary variable that takes value 1 only if $[u,v]$ ($\in\tilde{E}'$) is traversed by walk $t$                                               \\
$\sigma^{v_0}_{tuv}$                          & Continuous variable that indicate the volume of flow from artificial location $v_0$ passing through arc $(u,v)$ ($\in\tilde{E}'$) in walk $t$\\

\end{longtable}

\begin{table}
\begin{tabular}{p{2cm}p{4cm}p{9cm}}
Formulation     & Constraints         & Explanation  \\           
\midrule 
\multicolumn{3}{l}{For Analysis} \\

$P_{sub}$   & (19)-(25)           & feasible region of subaisle cuts                                                                     \\
$P_{basic}$ & (9)-(18)            & the basic formulation for the JOBPRP                                                                 \\
$P_A$     & (9)-(25)            & the basic formulation with subaisle cuts                                                             \\
$P_g$     & (26)-(42)           & a formulation which only force artificial locations to be in the same connected component            \\
$P_f$     & (26)-(34), (36)-(46)           & a flow-based formulation which only force artificial locations to be in the same connected component \\
$P_G$     & (19)-(42)           & a non-compact improved formulation for the JOBPRP                                                               \\
$P_F$    & (19)-(34), (36)-(46) & a flow-based improved formulation for the JOBPRP  \\

$P^1_U$    & (52)-(63) & a TSP-based no-reversal formulation for a single-block warehouse  \\
$P^2_U$    & (64)-(74) & a TSP-based no-reversal formulation for a 2-block warehouse  \\

\midrule
\multicolumn{3}{l}{For Experiment} \\

$P_O$      & - & $P_{basic}$ with the valid inequalities defined in \citep{VALLE2017817} and column inequalities \\
$P_G^+$  & - & $P_G$ with aisles
cuts, artificial vertex reversal constraints and column inequalities                                 \\
$P_F^+$  & - & $P_F$ with artificial vertex reversal constraints and column inequalities                                      \\
$P_U$   & - &  the no-reversal formulation in \citep{VALLE2017817} with column inequalities                                                  \\
$P^{1+}_U$ & - & $P^{1}_U$ with column inequalities                       \\
$P^{2+}_U$ & - & $P^{2}_U$ with column inequalities    

\end{tabular}
\end{table}

\begin{table}
\tbl{Detailed results for the basic cuts}
{\begin{tabular}{llcccccccccccccc} \toprule
 & &\multicolumn{2}{l}{$P_O +$ basic cuts} 
 & & & & \multicolumn{2}{l}{$P_G^+ +$ basic cuts}
 & & & & \multicolumn{2}{l}{$P_F^+ +$ basic cuts}\\ \cmidrule{3-6} \cmidrule{8-11} \cmidrule{13-16}
 $\Delta$ & O & T(seconds) & UB & LB & GAP(\%) & & T(seconds) & UB & LB & GAP(\%) & & T(seconds) & UB & LB & GAP(\%) \\ \midrule
5 & 5 & 0.34 & 346 & 346 & 0 & & 0.21 & 346 & 346 & 0 & & 0.4 & 346 & 346 & 0\\
& 10 & 6.96 & 578 & 578 & 0 & & 3.23 & 578 & 578 & 0 & & 66 & 578 & 578 & 0\\
& 15 & 26 & 650 & 650 & 0 & & 9.4 & 650 & 650 & 0 & & 46 & 650 & 650 & 0\\
& 16 & 406 & 766 & 766 & 0 & & 37 & 766 & 766 & 0 & & 292 & 766 & 766 & 0\\
& 17 & 239 & 802 & 802 & 0 & & 46 & 802 & 802 & 0 & & 281 & 802 & 802 & 0\\
& 18 & 2400 & 870 & 806 & 7.4 & & 95 & 840 & 840 & 0 & & 830 & 840 & 840 & 0\\
& 19 & 2020 & 856 & 856 & 0 & & 72 & 856 & 856 & 0 & & 605 & 856 & 856 & 0\\
& 20 & 2067 & 864 & 864 & 0 & & 75 & 864 & 864 & 0 & & 745 & 864 & 864 & 0\\
& 21 & 2400 & 902 & 846 & 6.2 & & 270 & 892 & 892 & 0 & & 2400 & 892 & 850 & 4.7\\
& 22 & 2400 & 892 & 886 & 1.1 & & 190 & 892 & 892 & 0 & & 1045 & 892 & 892 & 0\\
& 23 & 2400 & 918 & 868 & 5.5 & & 416 & 908 & 908 & 0 & & 1465 & 908 & 908 & 0\\
& 24 & 2400 & - & 733 & - & & 2400 & 1064 & 934 & 12.2 & & 2400 & 1076 & 820 & 23.8\\
& 25 & 2400 & 1112 & 890 & 20 & & 2400 & 1108 & 944 & 14.8 & & 2400 & 1120 & 839 & 25.1\\
& 30 & 2400 & 1232 & 844 & 31.5 & & 2400 & 1212 & 942 & 22.3 & & 2400 & 1202 & 862 & 28.3\\
   
\midrule
10 & 5 & 0.08 & 368 & 368 & 0 & & 0.12 & 368 & 368 & 0 & & 0.09 & 368 & 368 & 0\\
& 10 & 45 & 656 & 656 & 0 & & 8.67 & 656 & 656 & 0 & & 27 & 656 & 656 & 0\\
& 15 & 122 & 874 & 874 & 0 & & 36 & 874 & 874 & 0 & & 224 & 874 & 874 & 0\\
& 16 & 138 & 926 & 926 & 0 & & 65 & 926 & 926 & 0 & & 430 & 926 & 926 & 0\\
& 17 & 408 & 960 & 960 & 0 & & 97 & 960 & 960 & 0 & & 501 & 960 & 960 & 0\\
& 18 & 410 & 970 & 970 & 0 & & 151 & 970 & 970 & 0 & & 1055 & 970 & 970 & 0\\
& 19 & 300 & 978 & 978 & 0 & & 156 & 978 & 978 & 0 & & 1421 & 978 & 978 & 0\\
& 20 & 524 & 984 & 984 & 0 & & 166 & 984 & 984 & 0 & & 1640 & 984 & 984 & 0\\
& 21 & 1410 & 990 & 990 & 0 & & 281 & 990 & 990 & 0 & & 2192 & 990 & 990 & 0\\
& 22 & 2203 & 1000 & 1000 & 0 & & 125 & 1000 & 1000 & 0 & & 2400 & 1014 & 938 & 7.5\\
& 23 & 2400 & 1344 & 743 & 44.7 & & 2400 & 1132 & 1011 & 10.7 & & 2400 & 1152 & 926 & 19.6\\
& 24 & 2400 & 1184 & 911 & 23.1 & & 2400 & 1168 & 1058 & 9.4 & & 2400 & 1194 & 937 & 21.5\\
& 25 & 2400 & 1230 & 945 & 23.2 & & 2400 & 1220 & 991 & 18.8 & & 2400 & 1196 & 962 & 19.6\\
& 30 & 2400 & 1306 & 958 & 26.6 & & 2400 & 1286 & 1052 & 18.2 & & 2400 & 1288 & 954 & 25.9\\
\midrule
20 & 5 & 13 & 570 & 570 & 0 & & 1.29 & 570 & 570 & 0 & & 8.7 & 570 & 570 & 0\\
& 10 & 57 & 912 & 912 & 0 & & 47 & 912 & 912 & 0 & & 113 & 912 & 912 & 0\\
& 15 & 2400 & 1022 & 1012 & 1 & & 63 & 1022 & 1022 & 0 & & 385 & 1022 & 1022 & 0\\
& 16 & 2400 & 1200 & 1088 & 9.3 & & 377 & 1200 & 1200 & 0 & & 1607 & 1200 & 1200 & 0\\
& 17 & 2400 & 1282 & 1047 & 18.3 & & 1602 & 1250 & 1250 & 0 & & 2400 & 1250 & 1183 & 5.4\\
& 18 & 2400 & 1330 & 1093 & 17.8 & & 2400 & 1296 & 1216 & 6.2 & & 2400 & 1302 & 1160 & 10.9\\
& 19 & 2400 & 1342 & 1188 & 11.5 & & 1799 & 1304 & 1304 & 0 & & 2400 & 1322 & 1187 & 10.2\\
& 20 & 2400 & 1352 & 1115 & 17.5 & & 2400 & 1332 & 1266 & 5 & & 2400 & 1352 & 1121 & 17.1\\
& 21 & 2400 & 1620 & 903 & 44.3 & & 2400 & 1520 & 1193 & 21.5 & & 2400 & 1518 & 1036 & 31.8\\
& 22 & 2400 & 1820 & 966 & 46.9 & & 2400 & 1532 & 1180 & 23 & & 2400 & 1536 & 1107 & 27.9\\
& 23 & 2400 & - & 989 & - & & 2400 & 1598 & 1175 & 26.5 & & 2400 & 1618 & 1044 & 35.5\\
& 24 & 2400 & - & 916 & - & & 2400 & 1640 & 1185 & 27.7 & & 2400 & 1652 & 1097 & 33.6\\
& 25 & 2400 & 1736 & 987 & 43.1 & & 2400 & 1674 & 1164 & 30.5 & & 2400 & 1654 & 1086 & 34.3\\
& 30 & 2400 & - & 979 & - & & 2400 & 1934 & 1104 & 42.9 & & 2400 & 1940 & 970 & 50\\
\bottomrule
\end{tabular}}
\end{table}

\begin{table}
\tbl{Detailed results for the single traversing constraints}
{\begin{tabular}{llcccccccccccccc} \toprule
 & &\multicolumn{2}{l}{$P_O +$ single traversing} 
 & & & & \multicolumn{2}{l}{$P_G^+ +$ single traversing}
 & & & & \multicolumn{2}{l}{$P_F^+ +$ single traversing}\\ \cmidrule{3-6} \cmidrule{8-11} \cmidrule{13-16}
 $\Delta$ & O & T(seconds) & UB & LB & GAP(\%) & & T(seconds) & UB & LB & GAP(\%) & & T(seconds) & UB & LB & GAP(\%) \\ \midrule
5 & 5 & 0.17 & 346 & 346 & 0 & & 0.17 & 346 & 346 & 0 & & 0.38 & 346 & 346 & 0\\
& 10 & 5.3 & 578 & 578 & 0 & & 2.16 & 578 & 578 & 0 & & 27 & 578 & 578 & 0\\
& 15 & 18 & 650 & 650 & 0 & & 17 & 650 & 650 & 0 & & 36 & 650 & 650 & 0\\
& 16 & 223 & 766 & 766 & 0 & & 33 & 766 & 766 & 0 & & 128 & 766 & 766 & 0\\
& 17 & 1460 & 802 & 802 & 0 & & 40 & 802 & 802 & 0 & & 208 & 802 & 802 & 0\\
& 18 & 1681 & 840 & 840 & 0 & & 68 & 840 & 840 & 0 & & 781 & 840 & 840 & 0\\
& 19 & 2400 & 872 & 831 & 4.7 & & 78 & 856 & 856 & 0 & & 363 & 856 & 856 & 0\\
& 20 & 2400 & 888 & 755 & 15 & & 139 & 864 & 864 & 0 & & 701 & 864 & 864 & 0\\
& 21 & 2400 & 898 & 871 & 3 & & 161 & 892 & 892 & 0 & & 1880 & 892 & 892 & 0\\
& 22 & 2400 & 898 & 857 & 4.6 & & 180 & 892 & 892 & 0 & & 735 & 892 & 892 & 0\\
& 23 & 2400 & 908 & 887 & 2.3 & & 251 & 908 & 908 & 0 & & 1385 & 908 & 908 & 0\\
& 24 & 2400 & 1112 & 819 & 26.3 & & 2400 & 1062 & 933 & 12.1 & & 2400 & 1078 & 831 & 22.9\\
& 25 & 2400 & 1136 & 764 & 32.7 & & 2400 & 1112 & 941 & 15.4 & & 2400 & 1098 & 894 & 18.6\\
& 30 & 2400 & 1234 & 847 & 31.4 & & 2400 & 1194 & 962 & 19.4 & & 2400 & 1198 & 919 & 23.3\\
\midrule
20 & 5 & 0.07 & 368 & 368 & 0 & & 0.13 & 368 & 368 & 0 & & 0.2 & 368 & 368 & 0\\
& 10 & 33 & 656 & 656 & 0 & & 6.44 & 656 & 656 & 0 & & 30 & 656 & 656 & 0\\
& 15 & 162 & 874 & 874 & 0 & & 50 & 874 & 874 & 0 & & 313 & 874 & 874 & 0\\
& 16 & 296 & 926 & 926 & 0 & & 45 & 926 & 926 & 0 & & 253 & 926 & 926 & 0\\
& 17 & 398 & 960 & 960 & 0 & & 168 & 960 & 960 & 0 & & 447 & 960 & 960 & 0\\
& 18 & 1519 & 970 & 970 & 0 & & 91 & 970 & 970 & 0 & & 895 & 970 & 970 & 0\\
& 19 & 742 & 978 & 978 & 0 & & 195 & 978 & 978 & 0 & & 2400 & 978 & 958 & 2\\
& 20 & 297 & 984 & 984 & 0 & & 155 & 984 & 984 & 0 & & 1502 & 984 & 984 & 0\\
& 21 & 1015 & 990 & 990 & 0 & & 125 & 990 & 990 & 0 & & 1515 & 990 & 990 & 0\\
& 22 & 325 & 1000 & 1000 & 0 & & 200 & 1000 & 1000 & 0 & & 1449 & 1000 & 1000 & 0\\
& 23 & 2400 & 1178 & 923 & 21.6 & & 2400 & 1128 & 1006 & 10.8 & & 2400 & 1150 & 939 & 18.3\\
& 24 & 2400 & 1188 & 904 & 23.9 & & 2400 & 1162 & 1021 & 12.1 & & 2400 & 1182 & 904 & 23.5\\
& 25 & 2400 & 1192 & 875 & 26.6 & & 2400 & 1210 & 1039 & 14.1 & & 2400 & 1196 & 952 & 20.4\\
& 30 & 2400 & 1320 & 862 & 34.7 & & 2400 & 1296 & 1074 & 17.1 & & 2400 & 1268 & 962 & 24.1\\
\midrule
5 & 5 & 8.9 & 570 & 570 & 0 & & 1.1 & 570 & 570 & 0 & & 4.4 & 570 & 570 & 0\\
& 10 & 125 & 912 & 912 & 0 & & 21 & 912 & 912 & 0 & & 157 & 912 & 912 & 0\\
& 15 & 2400 & 1022 & 1008 & 1.4 & & 68 & 1022 & 1022 & 0 & & 315 & 1022 & 1022 & 0\\
& 16 & 2400 & 1238 & 1045 & 15.6 & & 530 & 1200 & 1200 & 0 & & 2394 & 1200 & 1200 & 0\\
& 17 & 2400 & 1282 & 938 & 26.8 & & 1266 & 1250 & 1250 & 0 & & 2400 & 1250 & 1190 & 4.8\\
& 18 & 2400 & 1318 & 1187 & 9.9 & & 1568 & 1288 & 1288 & 0 & & 2400 & 1308 & 1076 & 17.7\\
& 19 & 2400 & 1350 & 1114 & 17.5 & & 2400 & 1316 & 1251 & 4.9 & & 2400 & 1308 & 1168 & 10.7\\
& 20 & 2400 & 1360 & 1072 & 21.2 & & 2400 & 1332 & 1264 & 5.1 & & 2400 & 1358 & 1109 & 18.3\\
& 21 & 2400 & 1744 & 1014 & 41.9 & & 2400 & 1510 & 1187 & 21.4 & & 2400 & 1542 & 1094 & 29.1\\
& 22 & 2400 & 1590 & 999 & 37.2 & & 2400 & 1534 & 1206 & 21.4 & & 2400 & 1566 & 1116 & 28.7\\
& 23 & 2400 & 1742 & 1020 & 41.4 & & 2400 & 1606 & 1165 & 27.5 & & 2400 & 1618 & 1117 & 31\\
& 24 & 2400 & 1672 & 940 & 43.8 & & 2400 & 1660 & 1143 & 31.1 & & 2400 & 1664 & 1133 & 31.9\\
& 25 & 2400 & - & 917 & - & & 2400 & 1680 & 1191 & 29.1 & & 2400 & 1672 & 1060 & 36.6\\
& 30 & 2400 & - & 967 & - & & 2400 & 1896 & 1134 & 40.2 & & 2400 & 1944 & 987 & 49.2\\
\bottomrule
\end{tabular}}
\end{table}

\end{document}